\documentclass[11pt]{amsart}

\usepackage{latexsym,amssymb,amsmath,graphicx,
hyperref,diagbox,subfig,caption,verbatim,ulem}
\usepackage{tikz}
\usetikzlibrary{shapes}
\usetikzlibrary{arrows}
\usepackage{fp}
\usepackage{rotating}
\usepackage{longtable}
\usepackage{mathtools}
\usepackage{centernot}
\usepackage{blkarray}
\usepackage{pifont}
\usepackage{supertabular}
\usepackage{caption}
\newcommand{\cmark}{\ding{51}}
\newcommand{\xmark}{\ding{55}}
\newcommand{\gmark}{\cellcolor{gray}}
\usepackage{colortbl}
\usepackage{chngcntr}

\numberwithin{equation}{section}
\newtheorem{theorem}{Theorem}[section]
\newtheorem{lemma}[theorem]{Lemma}

\newtheorem{corollary}[theorem]{Corollary}
\newtheorem{conjecture}[theorem]{Conjecture}

\theoremstyle{definition}
\newtheorem{definition}[theorem]{Definition} 
\newtheorem{remark}[theorem]{Remark}
\newtheorem{example}[theorem]{Example}

\usepackage{color} 
\definecolor{rossred}{rgb}{1.0,0.25,0.66}

\begin{document}

%%%%%%%%%%%%%%%%%%%%%%%%%%%%%%%%%%%%%%%%%%%%%%%%%%%%%%%%%%%%%%%% 
 
\title{The van der Waerden simplicial Complex and its Lefschetz properties}
\thanks{Version: \today}

\author[N. Ragunathan]{Naveena Ragunathan}
\address[N. Ragunathan]{Department of Mathematics and Statistics\\
McMaster University, Hamilton, ON, L8S 4L8, Canada}
\email{ragunatn@mcmaster.ca}
 
\author[A. Van Tuyl]{Adam Van Tuyl}
\address[A. Van Tuyl]{Department of Mathematics and Statistics\\
McMaster University, Hamilton, ON, L8S 4L8, Canada}
\email{vantuyla@mcmaster.ca}


\begin{abstract}
The van der Waerden simplicial complex, denoted ${\tt vdw}(n,k)$,
is the simpicial complex whose facets correspond to the 
arithmetic progressions of length $k$ in the set $\{1,\ldots,n\}$.
We study the Lefschetz properties of the Artinian ring 
$A({\tt vdw}(n,k)) = K[x_1,\ldots,x_n]/(I_{{\tt vdw}(n,k)} + 
\langle x_1^2,\ldots,x_n^2\rangle)$
where  $I_{{\tt vdw}(n,k)}$ is the associated
Stanley--Reisner ideal. If $k=1,2$ or $n-1$, the ring 
$A({\tt vdw}(n,k))$ 
will have the
Weak Lefschetz Property for all $n > k$.  
When $k=3$, we classify the rings $A({\tt vdw}(n,3))$ that have the 
Weak Lefschetz Property when the characteristic is zero. We conjecture that $A({\tt vdw}(n,k))$ fails 
to have the Weak Lefschetz Property if $n \gg k \geq 3$ and $k$
odd. We also 
classify when ${\tt vdw}(n,k)$ is
a pseudo-manifold, which allows us to show that $A({\tt vdw}(n,k))$ satisfies the Weak Lefschetz Property in
some degrees by using a result of Dao and Nair.
\end{abstract}
 
\maketitle
%%%%%%%%%%%%%%%%%%%%%%%%%%%%%%%%%%%%%%%%%%%%%

\section{Introduction}

The study of the Strong and Weak Lefschetz Properties of 
graded Artinian rings is an eclectic mix of many areas of 
mathematics, including commutative algebra,
algebraic geometry, linear algebra,
combinatorics, and representation theory.
The origins of this topic can be traced back to work of 
Stanley \cite{S1980}, Watanabe \cite{W1985}, and Reid, Roberts and Roitman \cite{RRR1991}.  For introductions to
this area, see \cite{HMMNWW2008,MN2013}, and for a sample of
recent research, see \cite{AGFMN}.
In this paper, we consider the Lefschetz properties for a 
family of graded Artinian rings constructed from the 
van der Waerden simplicial complexes ${\tt vdw}(n,k)$,
a family of complexes introduced by
Ehrenborg, Govindaiah, Park, and Readdy\cite{EGPR2017} within the
context of combinatorics and number theory.

Let $R = K[x_1,\ldots,x_n]$ be a polynomial ring
over a field $K$ with characteristic denoted ${\rm char}(K)$, and suppose $I$ is a homogeneous ideal 
of $R$ such that $\sqrt{I} = 
\langle x_1,\ldots,x_n\rangle$.   The ring $A = R/I$ is 
a graded Artinian ring, and consequently, there exists
an integer $e$ such that $A = \bigoplus_{i=0}^e A_i$ where
each $A_i = (R/I)_i$, the $i$-th graded piece of $A$, is 
a finite dimensional vector space
over $K$.
The ring $A$ is said to have the  {\it Weak 
Lefschetz Property} (WLP) if for a general linear form
$\ell$, the multiplication 
map $$\times \ell: A_i \rightarrow A_{i+1}$$ 
has maximal rank, i.e., $\times \ell$ is 
either injective or surjective, for 
$i=0,\ldots,e-1$.  We say $A$ has the {\it WLP in degree $j$} if the map
$\times \ell: A_j \rightarrow A_{j+1}$ has maximal rank.
If the map $\times \ell^d: A_i \rightarrow A_{i+d}$
also has maximal rank for all $i,d \geq 0$, then
we say $A$ has the {\it Strong Lefschetz Property} (SLP).
One of the driving questions in the field is to
identify families of graded Artinian rings with either
the WLP or SLP. For example, it has long been conjectured
that all graded Artinian complete intersections have the SLP 
in characteristic zero \cite{RRR1991}.

Even when $I$ is a monomial ideal of $R$, determining
the WLP or SLP of $A = R/I$ is quite subtle and difficult; for
some work in this direction, see \cite{NAS,CJK2025,MMN,MNS}. Observe 
that for $R/I$ to be an Artinian ring when $I$ is a monomial ideal, 
for each $i=1,\ldots,n$, there
is a positive integer $a_i$ such that $x_i^{a_i}$ is a minimal generator of $I$. 
Thus, we can
write our monomial ideal as
$$I = J + \langle x_1^{a_1},\ldots,x_n^{a_n}\rangle$$
where $J$ is generated by the 
monomial generators of $I$ that are not pure powers. One
strategy to study the Lefschetz properties of $R/I$ is to leverage the 
properties of the monomial ideal $J$.  As an example of
this approach, Dao and Nair \cite{DN} 
deduced Lefschetz properties for $R/I$ under the extra assumption that
$J$ is the Stanley-Reisner ideal of a pseudo-manifold. 
A similar approach to study the associated simplicial
complex is found in  \cite{H2024,H2025}.
The papers \cite{CFHNVT,MNS,NT2024,T2021} apply the same
philosophy by assuming $J$ is
the edge ideal of a graph $G$ (a quadratic squarefree monomial ideal) and relate the WLP and SLP to properties of the graph $G$.

We use this approach to consider the
case that $J$ is the Stanley-Reisner
ideal of the {\it van der Waerden} simplicial complex. 
This complex, denoted ${\tt vdw}(n,k)$, 
is the simplicial complex on the vertex set $[n] = \{1,\ldots,n\}$ 
whose facets are the arithmetic
progressions of length $k$ in $[n]$. 
The name is inspired by van der Waerden's classical theorem \cite{vdw1927}
related to a colouring of the numbers $\{1,\ldots,n\}$ so that an
arithmetic progression of length $k$ is mono-coloured
(see \cite{BERG} for more).
The homotopy
type of these complexes 
were first studied by Ehrenborg, {\it et al.} \cite{EGPR2017}, 
while \cite{HVT2025,HVT} have developed some
of the combinatorial commutative algebra properties of ${\tt vdw}(n,k)$.

We expand our understanding of ${\tt vdw}(n,k)$ by 
studying the Lefschetz properties of
$$A({\tt vdw}(n,k)) := K[x_1,\ldots,x_n]/(I_{{\tt vdw}(n,k)}+ 
\langle x_1^2, \ldots, x_n^2 \rangle),$$
an Artinian  ring where $I_{{\tt vdw}(n,k)}$ is the 
Stanley-Reisner ideal of ${\tt vdw}(n,k)$.  We are able to
derive both positive and negative results.
In the positive direction, we prove the following results:
\begin{theorem} \label{maintheorem} 
Let $n > k\geq 1$ be integers and define $A({\tt vdw}(n,k))$ as above.
Then
\begin{enumerate}
    \item  $A({\tt vdw}(n,1))$ has the SLP for all $n \geq 2$
    if ${\rm char}(K) \neq 2$;
    \item $A({\tt vdw}(n,2))$ has the WLP for all $n \geq 3$
    if ${\rm char}(K) = 0$;
    \item $A({\tt vdw}(n,n-1))$ has the SLP for all $n \geq 2$ if 
    ${\rm char}(K) = 0$;
    \item $A({\tt vdw}(n,k))$ has the WLP in degrees $0$ and
    $1$ for all $n > k$ if ${\rm char}(K)=0$; and
    \item $A({\tt vdw}(n,k))$ has the WLP in degree $k$ if
    $\frac{n}{2} \leq k < n$ and $k \geq 3$ if ${\rm char}(K) =0$.
\end{enumerate}
\end{theorem}
\noindent
Our proofs for Theorem \ref{maintheorem} (1)-(5) make 
use of a number of ingredients, including
results of Migliore--Mir\'{o}-Roig--Nagel 
\cite{MMN} and
Lundqvist--Nicklasson \cite{LN} 
that
allows us to check the WLP by only 
considering a special linear form, and
Dao--Nair's work \cite{DN} related to pseudo-manifolds. 

While Theorem \ref{maintheorem} implies $A({\tt vdw}(n,k))$ has the WLP for
many values of $n$ and $k$, the ring $A({\tt vdw}(7,3))$ is our 
first example in this family that fails to have the WLP.
In fact, if we fix $k=3$, we classify exactly when the Artinian ring $A({\tt vdw}(n,3))$
has the WLP:

\begin{theorem}
    Suppose ${\rm char}(K) =0$ and $n > 3$.
    Then $A({\tt vdw}(n,3))$ has the WLP if and only if 
    $n=4,5,6$.
\end{theorem}
\noindent
Computational evidence suggests that if $k \geq 3$ is odd, then
almost all rings $A({\tt vdw}(n,k))$ fail to have the WLP.  At the end of 
the paper, we make a 
conjecture (Conjecture \ref{conjecture}) to this effect.

Our paper is structured as follows.  In Section \ref{sec.background}
we recall the relevant background results. 
Section \ref{sec.familieswithLP} describes some families $A({\tt vdw}(n,k))$
that have either the SLP or WLP.
In Section \ref{sec.vdw(n,3)}, we
classify for which $n$ the Artinian ring  $A({\tt vdw}(n,3))$
has the WLP.  Section \ref{sec.pseudomanifold} classifies
when ${\tt vdw}(n,k)$ is a pseudo-manifold.
We combine this classification with a result of Dao-Nair \cite{DN} to show that $A({\tt vdw}(n,k))$ has the WLP in some degrees.
Section \ref{sec.futuredirections} focuses on our conjectures,
while Appendix \ref{appendix} and \ref{appendixB} provide
proofs of technical results used in
Section \ref{sec.vdw(n,3)}.

%%%%%%%%%%%%%%%%%%%%%%%%%%%%%%%%%%%%%%%%%%%%%%

\section{Background}
\label{sec.background}

Let $R = K[x_1,\ldots,x_n]$ be a
polynomial ring over an infinite field $K$.  At this point, 
we make no  additional assumptions
on the characteristic of $K$, although certain results will 
require the characteristic to be zero.  We continue to use the terminology 
defined in the introduction.

\subsection{Lefschetz Properties}
We review some of the needed results about the WLP and SLP.  
As noted in the introduction, an Artinian ring has the WLP if there is a general linear form $\ell$ such that
the map $\times \ell: A_i \rightarrow A_{i+1}$ has maximal
rank for all $i \geq 0$.  

Gaining access to this general form is sometimes difficult for
an arbitrary Artinian ring.  However, since the Artinian rings
we consider have the form $R/I$ where $I$ is a monomial ideal,
Migliore, Mir\'{o}-Roig, and
Nagel \cite[Proposition 2.2]{MMN} (which assumes the field
is infinite) and
Lundqvist and Nicklasson \cite[Proposition 4.3]{LN} (which 
does not assume any condition on the cardinality of the field),
showed that to check if $R/I$ has the WLP, it suffices
to check if the map given by multiplication by 
a specific linear form
has maximal rank.  Precisely, 

\begin{theorem}
\label{BASIC}
Suppose $I \subseteq R = K[x_1,\ldots,x_n]$
is a monomial ideal such that $A = R/I$ is an Artinian ring.
Then $A$ has the WLP if and only if 
the map $\times \ell: A_i \rightarrow A_{i+1}$
where
$$\ell = x_1 + \cdots + x_n$$
has maximal rank for all $i \geq 0$.
\end{theorem}
The next result was first shown by Stanley \cite{S1980}, using the 
Hard Lefschetz Theorem. Watanabe \cite{W1985}, and Reid, Roberts, and Roitman \cite{RRR1991}  gave alternative proofs to this
foundational result.  
\begin{theorem}
\label{STAN}
Let $n \geq 1$ be an integer, and suppose
$a_1,\ldots,a_n$ are positive integers with $a_i \geq 2$ for
$i=1,\ldots,n$.
If the characteristic of $K$ is zero,
then the Artinian ring 
\[{K[x_1, x_2, \ldots, x_n]}/{\langle x_{1}^{a_1}, x_{2}^{a_2}, \ldots ,x_{n}^{a_n}\rangle}\] 
   is a complete intersection that has the SLP.
\end{theorem}
\noindent
We use this result for the special case that $a_1=\cdots=a_n =2$.  The
following remark shows that to check the WLP, we can
make an assumption on $A_i$ and $A_{i+1}$.

\begin{remark}\label{rmk:restricttononzero}
Suppose that we wish to show that $\times \ell:A_i \rightarrow A_{i+1}$ has maximal rank, and in addition, we also know that either
$A_i = 0$ or $A_{i+1} = 0$.  In either of these cases, the 
map $\times \ell$ automatically has maximal rank.  Thus, to 
verify if an Artinian ring $A$ has the WLP, we only need
to consider the maps $\times \ell: A_i \rightarrow A_{i+1}$ where
both $A_i$ and $A_{i+1}$ are non-zero.  Similarly, to check
if $A$ has the SLP, we can restrict to checking 
if all the maps $\times \ell^d:A_i \rightarrow A_{i+d}$ have maximal rank when $A_i$ and $A_{i+d}$ are non-zero.
\end{remark}
\subsection{Simplicial complexes and the van der Waerden simplicial complexes}
A \textit{simplicial complex} on a finite set of vertices 
$X =  \{x_{1}, x_{2}, ,\ldots,x_{n}\}$, denoted by $\Delta$, 
is a collection of subsets of $X$ 
such that: $(i)$ \{$x_{i}$\} $\in$ $\Delta$ for all $x_{i} \in X$, and
$(ii)$ if $S \in \Delta$ and $A \subseteq S$, then $A \in \Delta$.
Note that we often omit the empty set when listing out the elements of $\Delta$. While some authors sometimes omit $(i)$ when defining simplicial complexes, for our purpose we include this
requirement; this prevents us from considering 
degenerate cases.

The elements of $\Delta$ are called \textit{faces}. The faces that are maximal with respect to containment are called the \textit{facets} of $\Delta$. If $F_1, F_2, \ldots, F_s$ are the  facets of $\Delta$, then we write $\Delta = \langle F_1, F_2, \ldots, F_s \rangle $, and say $\Delta$ is generated by $F_1, F_2, \ldots, F_s$.  

The \textit{dimension} of a face $F$ $\in$ $\Delta$ is $\dim(F) = |F|-1$. Accordingly, the dimension of the empty set is $-1$. The \textit{dimension of $\Delta$}, denoted $\dim \Delta$, is the maximal dimension of all of its faces.
A simplicial complex is {\it pure} if all its facets have
the same dimension.
If $d = 
\dim \Delta$, then the \textit{$f$-vector} of $\Delta$ is the 
$(d+2)$-tuple $f(\Delta) = (f_{-1},f_0,f_1,\ldots,f_d)$ where
$f_i$ is the number of faces of dimension $i$ of $\Delta$.
We sometimes write $f_i(\Delta)$ for $f_i$ if we want to distinguish
the simplicial complex to which we are referring.
Note that $f_{-1} =1$.

The Stanley-Reisner ideal allows one to study $\Delta$
via commutative algebra.

\begin{definition}
Suppose that $\Delta$ is a simplicial complex on $X = \{x_1,\ldots,x_n\}$.
The \textit{Stanley--Reisner ideal} of $\Delta$ is the ideal generated by the non-faces of $\Delta$, that is
$$
I_{\Delta} = \langle x_{i_1}x_{i_2} \cdots x_{i_t} ~\vert~ 
\{x_{i_1}, x_{i_2}, \ldots, x_{i_t}\} \notin \Delta \rangle
$$
in the polynomial ring $R = K[x_1,\ldots,x_n]$.
\end{definition}

By construction, the Stanley-Reisner ideal
$I_\Delta$ is a square-free monomial ideal.   The process
can also be reversed; given a square-free monomial ideal
$I$, one can construct a simplicial complex.

For any simplicial complex $\Delta$ on $X = \{x_1,\ldots,x_n\}$, we
define the graded Artinian ring 
$$A(\Delta) = K[x_1,\ldots,x_n]/(I_\Delta+\langle x_1^2,\ldots,x_n^2\rangle).$$
Because $A(\Delta)$ is a standard graded Artinian ring,
then there exists a non-negative 
integer $e$ such that $A(\Delta) = \bigoplus_{i=0}^e A(\Delta)_i$,
where $A(\Delta)_i$ is the $i$-th graded piece
of $A(\Delta)$. In
fact, each $A(\Delta)_i$ is a $K$-vector space.  The next
result summarizes some well-known properties of these 
vector spaces.

\begin{theorem}\label{thm.propofA(Delta)}
Let $\Delta$ be a simplicial complex on $X = \{x_1,\ldots,x_n\}$
with $f$-vector $f(\Delta) = (f_{-1},f_0,\ldots,f_d)$.  Then
\begin{enumerate}
    \item $\dim_K A(\Delta)_i = f_{i-1}$ for $0 \leq i \leq d+1$, and 
    $\dim_K  A(\Delta)_i = 0$ otherwise.
    \item a $K$-basis for $A(\Delta)_i$ is 
    $$\{\overline{x_{j_1}\cdots x_{j_{i}}} ~|~ \{x_{j_1},\ldots,
    x_{j_{i}}\} \in \Delta \}$$
    where $\overline{x_{j_1}\cdots x_{j_{i}}}$ denotes the
    equivalence class of the monomial $x_{j_1}\cdots x_{j_{i}}$
    in $A(\Delta)$.
    \item $\times \ell:A(\Delta)_0  \rightarrow A(\Delta)_1$ has
    maximal rank, i.e., $A(\Delta)$ has the WLP in degree 0.
\end{enumerate}
\end{theorem}

\noindent
Observe that Theorem \ref{thm.propofA(Delta)} (2) implies that there 
is a one-to-one correspondence between the faces of dimension
$i-1$ of $\Delta$ and the basis elements of  $A(\Delta)_i$.

\begin{example}\label{ex.lingalgsetup}
By using the basis of Theorem \ref{thm.propofA(Delta)},
checking whether $A(\Delta)$ has the WLP can be reframed
as a linear algebra question by Theorem \ref{BASIC}.
Since our paper exploits this idea,
we illustrate this approach via an example.  Consider the 
simplicial complex $\Delta$ with facets 
$$\Delta = \langle \{x_1,x_2,x_3,x_4\},\{x_2,x_3,x_4,x_5\}
,\{x_3,x_4,x_5,x_6\}\rangle.$$
Suppose we wish to determine if $A(\Delta)$ has the WLP in degree
two, i.e., we want to check that the map 
$\times \ell:A(\Delta)_2 \rightarrow A(\Delta)_3$ 
has maximal rank for a generic
choice of $\ell$.  By Theorem \ref{BASIC}, it
suffices to check that the map $\times \ell:A(\Delta)_2 \rightarrow A(\Delta)_3$
has maximal rank for the linear form
$\ell = x_1+\cdots+x_6$.  By Theorem \ref{thm.propofA(Delta)},
a $K$-basis for $A(\Delta)_2$ is
$$\{\overline{x_1x_2},
\overline{x_1x_3},
\overline{x_2x_3},
\overline{x_1x_4},
\overline{x_2x_4},
\overline{x_3x_4},
\overline{x_2x_5},
\overline{x_3x_5},
\overline{x_4x_5},
\overline{x_3x_6},
\overline{x_4x_6},\overline{x_5x_6}\},$$
and a $K$-basis for $A(\Delta)_3$ is
$$\{
\overline{x_1x_2x_3},
\overline{x_1x_2x_4},
\overline{x_1x_3x_4},
\overline{x_2x_3x_4},
\overline{x_2x_3x_5},
\overline{x_2x_4x_5},
\overline{x_3x_4x_5},
\overline{x_3x_4x_6},
\overline{x_3x_5x_6},
\overline{x_4x_5x_6}\}.$$
Note that basis elements of $A(\Delta)_2$ (respectively,
$A(\Delta)_3$), are in one-to-one correspondence with the 
faces of dimension one (respectively dimension two), of
$\Delta$.
Consequently, after fixing these bases, the map
$\times \ell:A(\Delta)_2 \rightarrow A(\Delta)_3$ corresponds
to matrix multiplication given by the matrix
\begin{center}
\begin{blockarray}{ccccccccccccc}
$x_1x_2$ & $x_1x_3$ & $x_2x_3$ & $x_1x_4$ & $x_2x_4$ & $x_3x_4$ & $x_2x_5$ & $x_3x_5$ & $x_4x_5$ & $x_3x_6$ & $x_4x_6$ & $x_5x_6$\\
\begin{block}{[cccccccccccc]c}
1 & 1 & 1 & 0 & 0 & 0 & 0 & 0 & 0 & 0 & 0 & 0&$x_1x_2x_3$\\
1 & 0 & 0 & 1 & 1 & 0 & 0 & 0 & 0 & 0 & 0 & 0&$x_1x_2x_4$\\
0 & 1 & 0 & 1 & 0 & 1 & 0 & 0 & 0 & 0 & 0 & 0&$x_1x_3x_4$\\
0 & 0 & 1 & 0 & 1 & 1 & 0 & 0 & 0 & 0 & 0 & 0&$x_2x_3x_4$\\
0 & 0 & 1 & 0 & 0 & 0 & 1 & 1 & 0 & 0 & 0 & 0&$x_2x_3x_5$\\
0 & 0 & 0 & 0 & 1 & 0 & 1 & 0 & 1 & 0 & 0 & 0& $x_2x_4x_5$\\
0 & 0 & 0 & 0 & 0 & 1 & 0 & 1 & 1 & 0 & 0 & 0&$x_3x_4x_5$\\
0 & 0 & 0 & 0 & 0 & 1 & 0 & 0 & 0 & 1 & 1 & 0&$x_3x_4x_6$\\
0 & 0 & 0 & 0 & 0 & 0 & 0 & 1 & 0 & 1 & 0 &1 &$x_3x_5x_6$\\
0 & 0 & 0 & 0 & 0 & 0 & 0 & 0 & 1 & 0 & 1 &1 & $x_4x_5x_6$\\
\end{block}
\end{blockarray}
\end{center}
We have indexed the columns by
the basis elements of $A(\Delta)_2$ and indexed the rows by
the basis elements of $A(\Delta)_3$, but to simplify our
notation, we have dropped the ``bar'' denoting the equivalence
class of the monomial.    As an example of how this matrix 
is constructed, if we consider the basis element $x_1x_2 \in A(\Delta)_2$,
it is sent to the element:
$$x_1x_2 \mapsto  (x_1+\cdots+x_6)(x_1x_2) =
x_1^2x_2 + x_1x_2^2+x_1x_2x_3+x_1x_2x_4+x_1x_2x_5 +x_1x_2x_6\in A(\Delta)_3.$$
But in $A(\Delta)_3$, $x_i^2 =0$, and $x_1x_2x_5 = 0$ and 
$x_1x_2x_6=0$ since
$\{x_1,x_2,x_5\}$ and $\{x_1,x_2,x_6\}$ are both not in $ \Delta$. So,
when we write $\ell \times (x_1x_2)$ in terms of the basis elements of
$A(\Delta)_3$, we have  $$x_1x_2 \mapsto 1\cdot x_1x_2x_3+1\cdot x_1x_2x_4.$$
The ones in the column indexed  by $x_1x_2$ then record how to write
the image of $x_1x_2$ in terms of the basis of $A(\Delta)_3$.

To determine whether or not $A(\Delta)$ has the WLP in degree two now reduces to checking
if the rank of the matrix $M$ given above is as large as possible, 
that is, we need to check if
$${\rm rank}(M) = \min\{\dim_K A(\Delta)_2 =12,\dim_K A(\Delta)_3=10\} = 10.$$
Standard techniques in linear algebra can now verify that the rank
of this matrix is indeed ten, so $A(\Delta)$ has the WLP in degree two.
\end{example}

The following result of Dao and Nair \cite{DN} provides a tool
to check for the WLP in degree one.  The $1$-skeleton $G(\Delta)$ of $\Delta$ is the graph whose 
edges are the one-dimensional faces of $\Delta$.
\begin{theorem}[{\cite[Theorem 3.3]{DN}}]
\label{GSKEL}
Suppose that the characteristic of $K$ is $0$.
    Let $\Delta$ be a simplicial complex, and let $G(\Delta)$ be  its 1-skeleton. 
        If $\dim_{K} A(\Delta)_{2} \geq \dim_K A(\Delta)_{1}$, then $A(\Delta)$ has the WLP in degree one if and only if $G(\Delta)$ has no bipartite components.
       
\end{theorem}
\noindent
\begin{remark}
Dao--Nair \cite{DN} also comment that the above result holds if
the characteristic of $K$ is large enough.
\end{remark}

We now focus our attention on  $A(\Delta)$ when $\Delta$ is a van der Waerden simplicial complex, as first defined by 
Ehrenborg, {\it et al.} \cite{EGPR2017}.  We are primarily interested
in the Lefschetz properties of $A(\Delta)$ for this family.

\begin{definition}
    Fix integers $n>k \geq 1$. The 
    \textit{van der Waerden simplicial complex} on
    $X = \{x_1,\ldots,x_n\}$,
    denoted ${\tt vdw}(n,k)$,
    is the simplicial complex where each facet 
    $$\{x_i,x_{i+d},x_{i+2d},\ldots,x_{i+kd}\}$$ corresponds
    to an arithmetic progression $\{i,i+d,i+2d,\ldots,i+kd\}$ of 
    length $k$ on $\{1,\ldots,n\}$. 
\end{definition}

In the definition of ${\tt vdw}(n,k)$, the value of
$k$ refers to the number of steps, and thus each
facet has $k+1$ elements. So ${\tt vdw}(n,k)$ is pure with $\dim {\tt vdw}(n,k) = k$. One might find it easier to interpret the value of $k$ as the number of
``jumps".
Given a facet of ${\tt vdw}(n,k)$, say 
$F$ = {\{$x_{i}$, $x_{i+d}$,$\ldots$,$x_{i+kd}$}\}, 
we define $d$ to be the \textit{jump factor} of $F$, and denote it 
by ${\rm jump}(F)$.

\begin{example}
We construct the van der Waerden complex ${\tt vdw}(7,2)$.
If $n=7$ and $k=2$, then
$$\{\{1,2,3\},\{2,3,4\},\{3,4,5\},\{4,5,6\},
\{5,6,7\},\{1,3,5\},\{2,4,6\},\{3,5,7\},\{1,4,7\}\}$$
is the set of arithmetic progressions of length three (or
progressions with two ``jumps'')
on $\{1,\ldots,7\}$.  These progressions define the facets of
${\tt vdw}(7,2)$, namely
\begin{multline*}
{\tt vdw}(7,2) = \langle \{x_1,x_2,x_3\},
\{x_2,x_3,x_4\},\{x_3,x_4,x_5\},\{x_4,x_5,x_6\},
\{x_5,x_6,x_7\}, \\
\{x_1,x_3,x_5\},\{x_2,x_4,x_6\},\{x_3,x_5,x_7\},
\{x_1,x_4,x_7\}\rangle.
\end{multline*}
If we consider the facet $F = \{x_1,x_4,x_7\}$, then 
${\rm jump}(F) =3$, while the facet
$G = \{x_2,x_4,x_6\}$ has ${\rm jump}(G) = 2$.
\end{example}

%%%%%%%%%%%%%%%%%%%%%%%%%%%%%%%%%%%%%

\section{Families of $A({\tt vdw}(n,k))$ with the 
Strong and Weak Lefschetz Properties} 
\label{sec.familieswithLP}

In this section, we identify some families of rings
$A({\tt vdw}(n,k))$ that have either the SLP or the WLP.  Informally,
this will happen if either $k$ is small or large, when
compared to $n$.

We begin by showing that  
the ring  $A({\tt vdw}(n,k))$ always has the WLP in the 
first two degrees for any van der Waerden complex 
${\tt vdw}(n,k)$.

\begin{theorem} \label{Deg1}
Suppose that ${\rm char}(K) =0$.
For all integers $n > k \geq 1$, $A({\tt vdw}(n,k))$ has
the WLP in degree $0$ and $1$.
\end{theorem}

\begin{proof}
Let $A = A({\tt vdw}(n,k))$, and so by   
Theorem \ref{thm.propofA(Delta)} (3), $A$ has the WLP in 
degree 0.

For an arbitrary ${\tt vdw}(n,k)$, consider the following $n-k$ facets: 
\begin{eqnarray*}
F_t &=& \{x_t, x_{t+1}, \ldots, x_{t+k}\} ~~\mbox{for $t=1,\ldots,n-k$}. 
\end{eqnarray*}
From these facets, we see that $\{x_i,x_{i+1}\}$ is a one-dimensional
face of ${\tt vdw}(n,k)$ for $i=1,\ldots,n-1$.  These 
edges form a path in $G(\Delta)$ between
every vertex of the graph, so 
the $1$-skeleton $G(\Delta)$ is connected.

If $n=2$, then  $k=1$ and ${\tt vdw}(2,1)$ has exactly one facet,
namely $\{x_1,x_2\}$.  In this case $I_{{\tt vdw}(2,1)} =
\langle 0 \rangle$.  Consequently, 
$A({\tt vdw}(2,1)) = K[x_1,x_2]/\langle x_1^2,x_2^2 \rangle,$
and this ring has the WLP in general (and thus also specifically in degree 1) by Theorem \ref{STAN}.

On the other hand, if $n \geq 3$, then ${\tt vdw}(n,k)$
has at least $n$ faces of dimension one.  To see this,
we see that $F_1$ contributes the $k$ one-dimensional
faces 
$$\{x_1,x_{k+1}\},\{x_2,x_{k+1}\},\ldots,\{x_k,x_{k+1}\},$$
while $F_2,\ldots,F_{n-k}$ each contribute
$\{x_{k+i-1},x_{k+i}\}$ for $i=2,\ldots,n-k$.  This gives 
a total of $k+(n-k-1) = n-1$ faces of dimension one. 
But if $k=1$, then $\{x_1,x_3\}$ is also a facet of
${\tt vdw}(n,1)$, and if $k\geq 2$, then $\{x_1,x_3\}$ is also
a face of dimension one since it is a subset of $F_1$. 

 We thus have $\dim_K A_1 = n \leq \dim_K
A_2$.  The graph $G(\Delta)$ is connected,
and furthermore, $\{x_1,x_2\},\{x_1,x_3\}$, and $\{x_2,x_3\}$
are all faces of $\Delta$, so $G(\Delta)$ is not bipartite
since these three edges form a three cycle.  The conclusion
now follows from Theorem \ref{GSKEL}.
\end{proof}

\begin{example}
    The previous result fails if ${\rm char}(K) \neq 0$.  
    For example, consider $A({\tt vdw}(4,3))$. The matrix representation of the map $\times \ell: A_1 \rightarrow A_2$ is the following: 
\begin{center}
\begin{blockarray}{ccccc}
$x_1$ & $x_2$ & $x_3$ & $x_4$\\
\begin{block}{[cccc]c}
1 & 1 & 0 & 0 & $x_1x_2$\\
1 & 0 & 1 & 0 & $x_1x_3$\\
1 & 0 & 0 & 1 & $x_1x_4$\\
0 & 1 & 1 & 0 & $x_2x_3$\\
0 & 1 & 0 & 1 & $x_2x_4$\\
0 & 0 & 1 & 1 & $x_3x_4$\\
\end{block}
\end{blockarray}
\end{center}
This matrix has maximal rank 4 in characteristic 0. However, its rank drops to 3 in characteristic 2.
This drop in rank can be observed directly as the matrix has 15 possible $4 \times 4$ submatrices, of which 12 have determinant  $\pm 2$, and the remaining submatrices all have determinant 0.
\end{example}

We now have the tools to prove the following result.

\begin{theorem} \label{KISONE}
Suppose ${\rm char}(K) \neq 2$.
For all integers $n \geq 2$,  $A({\tt vdw}(n,1))$ has the SLP.
\end{theorem}

\begin{proof}
Let $A = A({\tt vdw}(n,1))$.
By Theorem \ref{thm.propofA(Delta)}, we have $A_i =0$ except
for $i=0,1,2$.  To show that $A$ has the SLP, by Remark \ref{rmk:restricttononzero} 
 the only maps we need to
check are: 1) $\times \ell: A_0 \rightarrow A_1$; 2) $\times \ell:A_1\rightarrow A_2$; and 3) $\times \ell^2: A_0 \rightarrow A_2$.  

By Theorem \ref{thm.propofA(Delta)}, case 1) holds for
any characteristic.  Case 2) holds by Theorem \ref{Deg1} if
we assume that ${\rm char}(K)=0$.  However, since we want our
result to hold for all ${\rm char}(K) \neq 2$, we give a different
proof that $\times \ell:A_1 \rightarrow A_2$ has maximal rank 
when ${\rm char}(K) \neq 2$. The facets of ${\rm vdw}(n,1)$ 
are $\{\{x_i,x_j\} ~|~ 1 \leq i < j \leq n \}$. So the matrix
representation of $\times\ell$ is an $\binom{n}{2} \times n$ matrix.
Consider the $n \times n$ submatrix whose rows are indexed
$\{x_1,x_2\},\ldots,\{x_1,x_n\}$ and $\{x_{n-1},x_n\}$. This
submatrix has the form
\begin{center}
\begin{blockarray}{ccccccc}
 $x_1$ & $x_2$ &  $x_3$ & $\cdots$  & $x_{n-1}$ & $x_n$ &\\
\begin{block}{[cccccc]c}
$1$ &   $1$    &  0&$\cdots$  & $0$ &    $0$  & $x_1x_2$\\
$1$ &   $0$    & 1 & $\cdots$  & $0$ & $0$  & $x_1x_3$\\
$\vdots$ & $\vdots$ & $\vdots$ & $\ddots$ & $\vdots$ & $\vdots$ \\
$1$ &   $0$    &$0$ &  $\cdots$  & $1$ & $0$& $x_1x_{n-1}$\\
$1$ &   $0$    &$0$ &  $\cdots$  & $0$ & $1$& $x_1x_n$\\
$0$ & $0$ & $0$ & $\cdots$ & $1$ & $1$  & $x_{n-1}x_{n}$\\
\end{block}
\end{blockarray}.
\end{center}
Using standard linear algebra techniques, this submatrix has 
rank $n$ when ${\rm char}(K) \neq 2$, so the 
matrix representation of $\times \ell$ has rank $n$. (Note 
that when 
${\rm char}(K) =2$,  this matrix does not have maximal rank
since there are two ones in each row, so the sum of all
the columns is the ${\bf 0}$ vector.)

For case 3), because $A_0 = K$, 
$\dim_K A_0 = 1$, and the map $\times \ell^2: A_0 \rightarrow A_2$
sends $1$ to $(x_1+\cdots+x_n)^2$ in $A_2$.  This is not
the zero map, so the map is injective, and hence this map also has maximal rank.  Thus $A$ has the SLP, as desired.
\end{proof}

\begin{example}
    It can be checked that $A({\tt vdw}(3,1))$ fails to have
    the SLP in characteristic two.  
\end{example}
When $k=2$, we can prove that $A({\tt vdw}(n,2))$ will always
have the WLP, although it may fail to have the SLP.   We first require
a lemma about the $f$-vector of ${\tt vdw}(n,2)$.

\begin{lemma}\label{lem.f-vector}
    For all integers $n\geq 3$, the $f$-vector 
    $f= (1,n,f_1,f_2)$ of ${\tt vdw}(n,2)$ satisfies
    $f_1 \geq f_2$.
\end{lemma}

\begin{proof}
    Set $\Delta(n) = {\tt vdw}(n,2)$.  We do induction on $n \geq 3$.
    For $n=3$, $\Delta(3) = \langle \{x_1,x_2,x_3\} \rangle$,
    so $f(\Delta(3)) = (1,3,3,1)$, and the result holds.

    We now assume $f_1(\Delta(n-1))\geq f_2(\Delta(n-1))$.  Set
    $d = \max\{j ~|~ 1 + 2j \leq n\}$.  Then $\Delta(n)$ has exactly
    $d$ facets that contain $x_n$, namely
    $$F_i = \{x_{n-2i},x_{n-i},x_n\}~~\mbox{for $i=1,\ldots,d$.}$$
    Since every facet of $\Delta(n-1)$ is also a facet of 
    $\Delta(n)$, we have 
    $f_2(\Delta(n)) = f_2(\Delta(n-1))+d.$
    These $d$ facets contribute at least $d$ one-dimensional faces
    to $\Delta(n)$; specifically,
    $$\{x_{n-1},x_n\},\{x_{n-2},x_n\},\ldots,\{x_{n-d},x_n\}.$$
    None of these one-dimensional faces belong to $\Delta(n-1)$,
    and thus $f_1(\Delta(n)) \geq f_1(\Delta(n-1))+d$.
    By induction, we then have
    $$f_1(\Delta(n)) \geq f_1(\Delta(n-1))+d \geq f_2(\Delta(n-1))+d = 
    f_2(\Delta(n)),$$
    as desired.
\end{proof}
We can prove that the following family of rings has the WLP.

\begin{theorem}
\label{KISTWO}
Suppose that ${\rm char}(K) =0$.
For all integers $n \geq 3$,  $A({\tt vdw}(n,2))$ has the WLP.
\end{theorem}

\begin{proof}
Let $A = A({\tt vdw}(n,2))$.  We know that $A_i \neq 0$
if and only if $i=0,1,2,3$ by Theorem \ref{thm.propofA(Delta)}.
By Theorem \ref{Deg1}, we know that $A$ has the 
WLP in degrees $0$ and $1$. By Remark \ref{rmk:restricttononzero}, 
it will suffice to show that $A$ has the WLP in degree two.

As in Example \ref{ex.lingalgsetup}, we can let
$\ell = x_1+\cdots+x_n$, and we can associate to
$\ell$ a matrix $M_\ell$ of size $(\dim_K A_3 \times \dim_K A_2).$  To complete the proof, we need
to verify that 
$${\rm rank}(M_\ell) = \min\{\dim_K A_2, \dim_K A_3\} = \dim_K A_3.$$  
The last equality follows from 
Lemma \ref{lem.f-vector} since $\dim_K A_2 = f_1 \geq  f_2 = \dim_K A_3$.

We proceed by induction on $n$.  
For $n=3$, the matrix representation of the map
$\times \ell:A_2 \rightarrow A_3$ has the form
\begin{center}
\begin{blockarray}{cccc}
$x_1x_2$ & $x_1x_3$ & $x_2x_3$\\
\begin{block}{[ccc]c}
1 & 1 & 1 & $x_1x_2x_3$\\
\end{block}
\end{blockarray}
\end{center}
It is immediate that this matrix has rank 1, the maximal rank.

We now prove the statement for $n$.  Our induction hypothesis
is that the matrix $M_{\ell'}$ that represents the map
$\times \ell':A({\tt vdw}(n-1,2))_2 \rightarrow A({\tt vdw}(n-1,2))_3$
has $${\rm rank}(M_{\ell'}) = \dim_K A({\tt vdw}(n-1,2))_3$$ where
$\ell' = x_1+\cdots+x_{n-1}$.

As noted within the proof of Lemma \ref{lem.f-vector}, ${\tt vdw}(n,2)$
has $d$ facets that contain $x_n$ where $d = \max\{j ~|~ 1 + 2j \leq n\}$.
In particular, these $d$ facets are
$$
F_i = \{x_{n-2i},x_{n-i},x_n\}~~\mbox{for $i=1,\ldots,d$.}
$$
The remaining facets of ${\tt vdw}(n,2)$ are also facets of 
${\tt vdw}(n-1,2)$. We order the basis
elements of $A({\tt vdw}(n,2))_3$ so that the last $d$ elements
correspond to the $d$ facets of ${\tt vdw}(n,2)$ that contain 
$x_n$.  We order the basis elements of $A({\tt vdw}(n,2))_2$
so that the last $d$ elements correspond to the one-dimensional
faces 
$$\{x_{n-d},x_n\},\{x_{n-d+1},x_n\},\ldots,\{x_{n-2},x_n\},\{x_{n-1},x_n\}.$$
Observe that we are making use of Theorem \ref{thm.propofA(Delta)} 
that says there is a one-to-one correspondence between the basis elements
of $A(\Delta)_i$ and the $(i-1)$-dimensional faces of $\Delta$.

With the bases ordered in this fashion, the matrix representation 
$M_\ell$ of the map 
$$\times \ell:A({\tt vdw}(n,2))_2 \rightarrow
A({\tt vdw}(n,2))_3$$ has the form
\begin{center}
\begin{blockarray}{ccccccc}
$x_1x_2$ &  $\cdots$ & $\cdots$ & $x_{n-d}x_n$ & $\cdots$  & $x_{n-1}x_n$ &\\
\begin{block}{[ccc|ccc]c}
  &       &           & 0 & $\cdots$   & 0      & $x_1x_2x_3$\\
  &    $M_1$  &       & $\vdots$ & $\ddots$ & $\vdots$  &  $\vdots$ \\
  &      &            & 0 & $\cdots$   & 0     & $x_{n-3}x_{n-2}x_{n-1}$\\
   \cline{1-6}
   &   &  &  &  &  & $x_{n-2d}x_{n-d}x_n$\\
   &  $M_2$  &  &  & $M_3$  &  & \vdots \\
   &   &  &  &  &  &  $x_{n-2}x_{n-1}x_n$\\
\end{block}
\end{blockarray}
\end{center}
where $M_1$, $M_2$, and $M_3$ are appropriately sized matrices.  

Observe
that $M_1$ has $M_{\ell'}$ as a submatrix.  Indeed, the indices
of the rows of $M_1$ correspond
to all the facets of ${\tt vdw}(n-1,2)$. Additionally, the 
indices of the columns of $M_1$ contain
all the faces of dimension one of ${\tt vdw}(n-1,2)$.  The new
facets containing $x_n$ may also contribute some new faces of dimension
one, so this is why $M_{\ell'}$ is a submatrix of $M_1$.  Thus $M_1$
is formed from $M_{\ell'}$ by adding additional columns.  By
adding new columns, to $M_{\ell'}$ to form $M_1$, we are not
changing the dimension of the row space.  So
$${\rm rank}(M_1) = {\rm rank}(M_{\ell'}) = \dim_K A({\tt vdw}(n-1,2))_3.$$

The matrix $M_3$ is a $d \times d$ matrix.  By 
multiplying $x_{n-i}x_n$ by $\ell$
for $i=1,\ldots,d$, 
and writing the product in terms of the basis of $A({\tt vdw}(n,3))_3$,
the matrix $M_3$ has 
the form
\begin{center}
\begin{blockarray}{cccccc}
 $x_{n-d}x_n$ & $x_{n-d+1}x_n$ &  $\cdots$  & $x_{n-2}x_n$ & $x_{n-1}x_n$ &\\
\begin{block}{[ccccc]c}
$1$ &   $0$    &  $\cdots$  & $0$ &    $0$  & $x_{n-2d}x_{n-d}x_{n}$\\
$\star$ &   $1$    &  $\cdots$  & $0$ & $0$  & $x_{n-2d+1}x_{n-d+1}x_{n}$\\
$\vdots$ & $\vdots$ & $\ddots$ & $\vdots$ & $\vdots$ & $\vdots$ \\
$\star$ &   $\star$    &  $\cdots$  & $1$ & 0 & $x_{n-4}x_{n-2}x_{n}$\\
$\star$ & $\star$ & $\cdots$ & $\star$ & 1  & $x_{n-2}x_{n-1}x_{n}$\\
\end{block}
\end{blockarray}
\end{center}
where $\star$ is either a $0$ or $1$.  From this description of
$M_3$, we see that ${\rm rank}(M_3) =d$, and consequently,
the row rank of the matrix consisting of only the last $d$ rows
of $M_{\ell}$ has rank $d$.

Because of the block matrix form of $M_\ell$, the last $d$ rows
cannot be in the row space of the first 
$\dim_K A({\tt vdw}(n-1,2))_3$ rows of
$M_\ell$.  Consequently,
$${\rm rank}(M_\ell) = {\rm rank}(M_\ell')+ d = 
\dim_K A({\tt vdw}((n-1,2))_3+d.$$
But $\dim_K A({\tt vdw}(n-1,2))_3+d = 
\dim_K A({\tt vdw}(n,2))_3$, so ${\rm rank}(M_\ell)$
has the desired rank.
\end{proof}

\begin{example}
    Unlike the case when $k=1,$ $A({\tt vdw}(n,2))$ does not always also have the SLP. In particular, $A({\tt vdw}(5,2))$ does not have the SLP  
    since the map $\times \ell^2:A_1 \rightarrow A_3$
    fails to have maximal rank. The associated matrix for this map is:
\[
\begin{blockarray}{ccccc c}
x_1 & x_2 & x_3 & x_4 & x_5 &\\
\begin{block}{[ccccc]c}
1 & 1 & 1 & 0 & 0 & x_1x_2x_3\\
1 & 0 & 1 & 0 & 1 & x_1x_3x_5\\
0 & 1 & 1 & 1 & 0 & x_2x_3x_4\\
0 & 0 & 1 & 1 & 1 & x_3x_4x_5\\
\end{block}
\end{blockarray}
\]
This matrix is not of maximal rank (it has rank three instead of the maximal rank four), and so $A({\tt vdw}(n,2))$ does not have the SLP.
\end{example}

We end this section with another family that
has the SLP.   This result is a consequence of Stanley's 
result (c.f. Theorem \ref{STAN}).

\begin{theorem} \label{NMINUS1}
   Suppose that ${\rm char}(K) =0$. 
   Then $A({\tt vdw}(n,n-1))$ has the SLP for all integers $n \geq 2$.
\end{theorem}

\begin{proof}
For all $n \geq 2$, the complex ${\tt vdw}(n,n-1)$ has only one facet,
namely $\{x_{1}, x_{2}, x_{3}, \ldots, x_{n}\}$.  Thus
$I_{{\tt vdw}(n,n-1)} = \langle 0 \rangle$, and hence
\[A({\tt vdw}(n,n-1)) = K[x_1,\ldots,x_n]/(\langle 0 \rangle + \langle x_{1}^2, x_{2}^2, \ldots, x_{n}^2\rangle).\]
The conclusion then follows from Theorem \ref{STAN}.
\end{proof}

\begin{remark}
Note that our assumption that $K$ is an infinite field of characteristic zero is an essential condition in order to apply Stanley's result.
\end{remark}

%%%%%%%%%%%%%%%%%%%%%%%%%%%%%%%%%%%%%%%%%%%%%%%%%%%%%%%%%%%%%%%%
\section{${\tt vdw}(n,3)$ and the failure of the Weak Lefschetz Property}
\label{sec.vdw(n,3)}

The previous section showed that if $k=1$ or $2$, then 
$A({\tt vdw}(n,k))$ always has the WLP.  However, when
we consider the case $k=3$, the main result of this section is
to show that $A({\tt vdw}(n,3))$ rarely has the WLP.  In fact,
we classify the $n$ for which this ring has the WLP
when the characteristic is $0$.

For this section, we will require the following lemma: 

\begin{lemma}\label{lem.f-vector-vdw(n,3)}
 For all integers $n\geq 8$, the $f$-vector 
    $f= (1,n,f_1,f_2,f_3)$ of ${\tt vdw}(n,3)$ satisfies
    $f_1 \leq f_2$.
\end{lemma}

Because the proof of this lemma is technical, we 
have postponed  the necessary details to Appendix \ref{appendix}.  We require the hypothesis $n \geq 8$ 
since the conclusion is false if $4 \leq n \leq 7$,
as shown in Table \ref{PLS} at the end of this paper.

We now prove the main result of this section.

\begin{theorem}\label{thm.casek=3}
    Suppose that ${\rm char}(K) = 0$, and let $n > 3$ be an integer.
    Then $A({\tt vdw}(n,3))$ has
    the WLP if and only if $n=4,5,6$.  
\end{theorem}

\begin{proof}
    If $A = A({\tt vdw}(n,3))$, then $A_i = 0$ for $i \geq 5$.  
    By Theorem \ref{Deg1}, the map $A$ has the WLP
    in degrees $0$ and $1$.  So, by Remark \ref{rmk:restricttononzero}, we only need to consider
    the maps
    $$\times \ell:A_2 \rightarrow A_3 ~~~\mbox{and} \times \ell:A_3 \rightarrow A_4.$$

    If $n=4$, then $A({\tt vdw}(4,3))$ has the WLP (in fact, the SLP), by
    Theorem \ref{NMINUS1}.  
For $n=5$ and $6$, we can apply the strategy of Example 
    \ref{ex.lingalgsetup} to check that the matrix representation of each
    of these maps has the maximum rank.  For completeness, the 
    matrices and these routine calculations 
    are included in Appendix \ref{appendixB}.

For $n \geq 7$, we will show that $A({\tt vdw}(n,3))$ fails to have
the WLP in degree two.  We need to treat the case $n=7$ separately 
because of Lemma \ref{lem.f-vector-vdw(n,3)}.  Specifically,
for $n \geq 8$, we will exploit the fact that $f_1({\tt vdw}(n,3)) \leq
f_2({\tt vdw}(n,3))$.  When $n=7$, we will require a different 
argument since
$f_1({\tt vdw}(7,3)) = 18 > 17 = f_2({\tt vdw}(7,3))$.

When $n=7$, the matrix representation $M$ of the map $\times \ell:A_2 \rightarrow A_3$ with respect to the basis of Theorem
\ref{thm.propofA(Delta)} where $\ell = x_1+ \cdots+ x_n$ is given by
\begin{center}
\resizebox{\textwidth}{!}
{
\begin{blockarray}{ccccccccccccccccccc}
 $x_1x_2$ & $x_1x_3$ & $x_1x_4$ &$ x_1x_5$ & $x_1x_7$ & $x_2x_3$ & $x_2x_4$ & $x_2x_5$ & $x_3x_4$ & $x_3x_5$ & $x_3x_6$ & $x_3x_7$ & $x_4x_5$ & $x_4x_6$ & $x_4x_7$ & $x_5x_6$ & $x_5x_7$ & $x_6x_7$&\\
\begin{block}{[cccccccccccccccccc]c}
1 & 1 & 0 & 0 & 0 & 1 & 0 & 0 & 0 & 0 & 0 & 0 & 0 & 0 & 0 & 0 & 0 & 0 & $x_1x_2x_3$\\
1 & 0 & 1 & 0 & 0 & 0 & 1 & 0 & 0 & 0 & 0 & 0 & 0 & 0 & 0 & 0 & 0 & 0 & $x_1x_2x_4$\\
0 & 1 & 1 & 0 & 0 & 0 & 0 & 0 & 1 & 0 & 0 & 0 & 0 & 0 & 0 & 0 & 0 & 0 & $x_1x_3x_4$\\
0 & 1 & 0 & 1 & 0 & 0 & 0 & 0 & 0 & 1 & 0 & 0 & 0 & 0 & 0 & 0 & 0 & 0 & $x_1x_3x_5$\\
0 & 1 & 0 & 0 & 1 & 0 & 0 & 0 & 0 & 0 & 0 & 1 & 0 & 0 & 0 & 0 & 0 & 0 & $x_1x_3x_7$\\
0 & 0 & 0 & 1 & 1 & 0 & 0 & 0 & 0 & 0 & 0 & 0 & 0 & 0 & 0 & 0 & 1 & 0 & $x_1x_5x_7$\\
0 & 0 & 0 & 0 & 0 & 1 & 1 & 0 & 1 & 0 & 0 & 0 & 0 & 0 & 0 & 0 & 0 & 0 & $x_2x_3x_4$\\
0 & 0 & 0 & 0 & 0 & 1 & 0 & 1 & 0 & 1 & 0 & 0 & 0 & 0 & 0 & 0 & 0 & 0 & $x_2x_3x_5$\\
0 & 0 & 0 & 0 & 0 & 0 & 1 & 1 & 0 & 0 & 0 & 0 & 1 & 0 & 0 & 0 & 0 & 0 & $x_2x_4x_5$\\
0 & 0 & 0 & 0 & 0 & 0 & 0 & 0 & 1 & 1 & 0 & 0 & 1 & 0 & 0 & 0 & 0 & 0 & $x_3x_4x_5$\\
0 & 0 & 0 & 0 & 0 & 0 & 0 & 0 & 1 & 0 & 1 & 0 & 0 & 1 & 0 & 0 & 0 & 0 & $x_3x_4x_6$\\
0 & 0 & 0 & 0 & 0 & 0 & 0 & 0 & 0 & 1 & 1 & 0 & 0 & 0 & 0 & 1 & 0 & 0 & $x_3x_5x_6$\\
0 & 0 & 0 & 0 & 0 & 0 & 0 & 0 & 0 & 1 & 0 & 1 & 0 & 0 & 0 & 0 & 1 & 0 & $x_3x_5x_7$\\
0 & 0 & 0 & 0 & 0 & 0 & 0 & 0 & 0 & 0 & 0 & 0 & 1 & 1 & 0 & 1 & 0 & 0 & $x_4x_5x_6$\\
0 & 0 & 0 & 0 & 0 & 0 & 0 & 0 & 0 & 0 & 0 & 0 & 1 & 0 & 1 & 0 & 1 & 0 & $x_4x_5x_7$\\
0 & 0 & 0 & 0 & 0 & 0 & 0 & 0 & 0 & 0 & 0 & 0 & 0 & 1 & 1 & 0 & 0 & 1 & $x_4x_6x_7$\\
0 & 0 & 0 & 0 & 0 & 0 & 0 & 0 & 0 & 0 & 0 & 0 & 0 & 0 & 0 & 1 & 1 & 1 & $x_5x_6x_7$\\
\end{block}
\end{blockarray}
}
\end{center}
Now set $x^T$ to  be the $17 \times 1$ column vector that is formed by taking the
transpose of the vector 
\setcounter{MaxMatrixCols}{20}
\[
x = \begin{bmatrix}
-1 & 1 & -1 & 1 & 1 & -1 & 0 & 1 & -1 & 0 & 1 & -1 & -1 & 0 & 1 & -1 & 1
\end{bmatrix}.
\]
The vector $x^T$ is in the kernel of $M^T$, the transpose of the 
matrix given above.  The matrix $M^T$ describes a map
$\varphi:A_3 \rightarrow A_2$.  
By the rank-nullity theorem, $17 = \dim_K A_3  = {\rm rank}(M^T) + 
\dim_K {\rm Nul}(M^T)$. Since $\dim_K {\rm Nul}(M^T) \geq 1$,
this gives ${\rm rank}(M^T) \leq 16$.  Since ${\rm rank}(M) = 
{\rm  rank}(M^T)$, this now shows that $A({\tt vdw}(7,3))$ does not have
the WLP in degree 2 since 
$${\rm rank}(M) \leq 16 < \min\{\dim_K A_2 = 18, \dim_K A_3 =17\} = 17.$$

We now suppose that $n \geq 8$.   We will show that $A = A({\tt vdw}(n,3))$ always fails to have the WLP in degree two, i.e.,
the map $\times \ell:A_2 \rightarrow A_3$ fails to have maximal
rank.  Because $\dim_K A_2 = f_1 \leq f_2 = \dim_k A_3$ for all $n \geq 8$
by Lemma \ref{lem.f-vector-vdw(n,3)}, it will suffice to
show that the map $\times \ell$ always fails to be injective. 
If $M$ denotes the matrix associated to the map $\times \ell$ 
with the basis as in Theorem \ref{thm.propofA(Delta)}, we
need to show that the null space of $M$ is non-empty.

The matrix $M$ is an $f_2 \times f_1$ matrix whose columns, respectively rows,  are 
indexed by the basis elements of $A_2$, respectively $A_3$.  
Consider an ordered basis elements of $A_2$, say 
$$m_1 \geq m_2 \geq \cdots \geq m_{f_1},$$
and order the columns using this order.
Because each $m_i$ corresponds to a basis element of $A_2$, each
$m_i = x_{a_i}x_{b_i}$ with $a_i < b_i$ and $\{x_{a_i},x_{b_i}\}$ a 
one-dimensional face of ${\tt vdw}(n,3)$.  With this ordering
of the basis elements, define an $f_1 \times 1$ column vector
${\bf z}$ of the form 
$$ 
{\bf z}^T = \begin{bmatrix} z_1 & z_2 & \cdots & z_{f_1} \end{bmatrix}^T.
$$
where  $$z_i = (-1)^{k_i+1}2^{k_i}  ~~\mbox{where $b_i-a_i = 2^{k_i}t_i$ and
$2\nmid t_i$}.$$
In other words, in the location indexed by the basis element
$m_i = x_{a_i}x_{b_i}$, we assign the largest power of $2$ that divides $b_i-a_i$, up to some sign.  Because our
characteristic is zero, every $z_i \neq 0$.

We now claim that $M{\bf z} = {\bf 0}$, from which our desired conclusion
will follow.   To prove this claim, we need to show that for
any row ${\bf r}$ of $M$, the dot product ${\bf r}\cdot {\bf z} =0$.

Consider any row ${\bf r}$ of $M$, and suppose that this row ${\bf r}$ is indexed 
by the monomial $x_ax_bx_c$ with $a<b< c$, which also corresponds to the two-dimensional face $\{x_a,x_b,x_c\}$.  The row vector
${\bf r}$ will consist
of all zeroes except three ones, which are located in the columns
indexed by $x_ax_b$, $x_ax_c$ and $x_bx_c$, i.e.,
\[
\begin{blockarray}{ccccccc c}
\cdots & x_ax_b & \cdots  &  x_ax_c & \cdots  & x_bx_c &  \cdots   & \\
\begin{block}{[ccccccc]c}
0 & 1 & 0 & 1 & 0 & 1 & 0 & x_ax_bx_c\\
\end{block}
\end{blockarray}
\]
This follows from the fact that $(x_1+\cdots +x_n)x_sx_t$ will
involve the basis element $x_ax_bx_c$ if and only if $x_sx_t$ divides
$x_ax_bx_c$.   

If $b-a = 2^{k_1}t_1$, $c-a = 2^{k_2}t_2$ and $
b-c = 2^{k_3}t_3$ where $2\nmid t_i$ for $i=1,2,3$, then
\begin{equation}\label{dotproduct}
{\bf r}\cdot{\bf z} = (-1)^{k_1+1}2^{k_1}+(-1)^{k_2+1}2^{k_2} + (-1)^{k_3+1}2^{k_3}.
\end{equation}
We thus need to determine the values of $k_1,k_2$ and $k_3$.

The face $\{x_a,x_b,x_c\}$ belongs to some
facet $\{x_i,x_{i+j},x_{i+2j},x_{i+3j}\}$ of ${\tt vdw}(n,3)$  for some $i$ and $j$.
So,  $\{x_a,x_b,x_c\}$ has one of the following
forms
$$\{x_i,x_{i+j},x_{i+2j}\},~\{x_i,x_{i+j},x_{i+3j}\},~\{x_i,x_{i+2j},
x_{i+3j}\},~~\mbox{or}~~\{x_{i+j},x_{i+2j},x_{i+3j}\}.$$

We now consider each possibility.  If 
$\{x_a,x_b,x_c\}=  \{x_i,x_{i+j},x_{i+2j}\}$, then $b-a = c-b = j$,
and $c-a=2j$.  So, if $j = 2^kt$, then $2j = 2^{k+1}t$.  So
in \eqref{dotproduct}, we have $k_1=k_3 = k$ and $k_2 =k+1$
and thus
$$(-1)^{k+1}2^{k}+(-1)^{k+2}2^{k+1} + (-1)^{k+1}2^{k} = (-1)^{k+1}2\cdot 2^k + (-1)^{k+2}2^{k+1} = 0.$$

If $\{x_a,x_b,x_c\}= \{x_i,x_{i+j},x_{i+3j}\}$,
we have $b-a = j$, $c-a = 3j$, and $c-b=2j$.  So, if $j = 2^kt$, we 
have $2j = 2^{k+1}t$ and $3j = 2^k(3t).$  
So in \eqref{dotproduct}, we have $k_1=k_2 = k$ and $k_3 =k+1$ and the
result follows as in the previous case.

If $\{x_a,x_b,x_c\}= \{x_i,x_{i+2j},x_{i+3j}\}$, then
$b-a = 2j$, $c-a=3j$, and $c-b =j$.  So, if $j =2^kt$, then
$2j =2^{k+1}t$ and $3j =2^k(3t)$.  Thus, in \eqref{dotproduct},
we have $k_2 = k_3 = k$ and $k_1=k+1$.  And thus
$$(-1)^{k+2}2^{k+1}+(-1)^{k+1}2^{k} + (-1)^{k+1}2^{k} = (-1)^{k+1}2\cdot 2^k + (-1)^{k+2}2^{k+1} = 0.$$

Finally, in the case that $\{x_a,x_b,x_c\}= \{x_{i+1},x_{i+2j},x_{i+3j}\}$
we have $b-a = c-b = j$ and $c-a=2j$, which is the same
as in the case $\{x_a,x_b,x_c\}= \{x_i,x_{i+j},x_{i+2j}\}$.

This now completes the proof. 
%Note that when $n \geq 7$,
%we showed that $A = A({\tt vdw}(n,3))$ fails to have the WLP in
%degree two.
\end{proof}

\begin{remark}
    In our proof, we explicitly constructed a vector
    ${\bf z}$ such that $M{\bf z} = {\bf 0}$.  
    Because the characteristic is $0$, this vector ${\bf z}$ we constructed
    is not the zero-vector.  If ${\rm char}(K) =2$, then 
    computations show that $A({\tt vdw}(n,3))$ fails to have
    the WLP for all $n \geq 4$, not just $n \geq 7$.
\end{remark}

%%%%%%%%%%%%%%%%%%%%%%%%%%%%%%%%%%%%%%%%%%%%%%%%%%%%%%%%%%%%%%%%
%%%%%%%%%%%%%%%%%%%%%%%%%%%%%%%%%%%%%%%%%%%%%%%%%%%%%%%%%%%%%%%%
%%%%%%%%%%%%%%%%%%%%%%%%%%%%%%%%%%%%%%%%%%%%%%%%%%%%%%%%%%%%%%%% 
\section{Pseudo-manifolds}
\label{sec.pseudomanifold}

In this section we classify
when ${\tt vdw}(n,k)$ is a pseudo-manifold (with
or without boundary).  
Our classification allows us to use results of Dao and Nair \cite{DN} to deduce when our algebras $A({\tt vdw}(n,k))$ have
some Lefschetz properties.  Recall the definition
of a pseudo-manifold.

\begin{definition}
\label{PMDEF}
    A $k$-dimensional simplicial complex is a \textit{pseudo-manifold} if
    \begin{enumerate}
        \item the simplicial complex is pure;
        \item all $(k-1)$-dimensional faces are contained in a maximum of two facets; and
        \item given any two facets, $F_a, F_b$ of the simplicial complex, there exists a sequence of facets $F_a=G_0, G_1, G_2, $\ldots, $G_n = F_b$ such that the dimension of $G_i$ $\cap$ $G_{i+1}$ is $(k-1)$ for all $0 \leq i \leq n-1$. 
    \end{enumerate}
    A $k$-dimensional pseudo-manifold has a \textit{boundary} if it has at least one $(k-1)$-dimensional face that belongs to exactly one facet. 
\end{definition}

We first establish some background:

\begin{lemma}
\label{X1}
Suppose $\{x_{i_1},\ldots,x_{i_{k+1}}\}$ is
a facet of ${\tt vdw}(n,k)$ 
on $\{x_1,\ldots,x_n\}$.  Then the indices
of the facet, i.e., $\{i_1,\ldots,i_{k+1}\}$,
satisfy one of the four conditions:
\begin{enumerate}
    \item all $k+1$ indices are  odd, or
    \item $\left\lceil {(k+1)/{2}} \right\rceil$ indices are odd  and $\left\lfloor {(k+1)/{2}} \right\rfloor$ indices are even, or
    \item $\left\lfloor {(k+1)/{2}} \right\rfloor$ 
    indices are odd and $\left\lceil {(k+1)/{2}} \right\rceil$ indices are even, or
    \item all $k+1$ indices are even.
\end{enumerate}
\end{lemma}

\begin{proof}
    An arbitrary facet of ${\tt vdw}(n,k)$ 
    has the form 
    $F = \{x_{i}, x_{i+a}, x_{i+2a}, x_{i+3a}, \ldots, x_{i+ka}\},$
   for some integers and $i$ and $a$.  The 
   conclusion then depends upon the parity
   of $i$ and $a$.  

   As an example, if $i$ is odd and $a$ is even,
   then $\{i,i+a,i+2a,\ldots,i+ka\}$ are all odd.
   On the other hand, 
   if $i$ and $a$ are both odd, we have
   $\{i,i+2a,i+4a,\ldots,\}$ are all odd
   and $\{i+a,i+3a,\ldots\}$ are all even.  
   The number in each set is either $\lfloor \frac{k+1}{2} \rfloor$ or $\lceil\frac{k+1}{2}\rceil$, where the number will depend upon 
   the parity of $k+1$.   
   
   The remaining cases
   are proved similarly.
 \end{proof}

Recall that if $F  = \{x_i,x_{i+a},\ldots,x_{i+ka}\}$ is a 
facet of ${\tt vdw}(n,k)$, then 
the \textit{jump factor} of $F$ is the common
difference $a$, i.e., 
 ${\rm jump}(F)=a$.
\begin{lemma}
\label{JUMPz}    
    If $\frac{n}{2} \leq k \leq$ n, then ${\rm jump}(F) = 1$ for all facets 
    $F \in {\tt vdw}(n,k)$. 
\end{lemma} 

\begin{proof}
    Let  $F = {\{x_{i}, x_{i+a},\ldots, x_{i+ka}}\}$ with $i+ka \leq n$ be a facet
    of ${\tt vdw}(n,k)$.  
    Suppose that $a \geq$ 2. Then $i+ka \geq i +2k$. But this means 
    $i+ka \geq i  +  \frac{2n}{2} > n$. 
    This gives is a contradiction, and thus  
    ${\rm jump}(F) = 1$ for all
    facets $F$.
    \end{proof}

We now come to the first main result of this
section.  The cases $k=1$ and $k=2$ are done separately.

\begin{theorem}
\label{X2}
    Suppose $n>k\geq 3$. Then ${\tt vdw}(n,k)$ is a pseudo-manifold if and only if $\frac{n}{2}$ $\leq$ k $<$ n.  Furthermore, 
    when ${\tt vdw}(n,k)$ is a pseudo-manifold, it has a
    boundary.
\end{theorem}
\begin{proof}

We begin with the backward direction. Let $\frac{n}{2}$ $\leq$ $k$ $\leq$ $n$, and consider ${\tt vdw}(n,k)$. By construction, 
${\tt vdw}(n,k)$ is pure.
By Lemma \ref{JUMPz},
all facets $F$ of ${\tt vdw}(n,k)$ have ${\rm jump}(F) = 1.$
Thus all the facets of ${\tt vdw}(n,k)$ are
$$F_i = \{x_i, x_{i+1}, \ldots, x_{k+i}\} ~~\mbox{for $i=1,\ldots,n-k$}.$$

Since ${\tt vdw}(n,k)$ has dimension $k$, let $G = \{x_{j_1},
\ldots,x_{j_k}\}$ be any face of dimension $k-1$.  We have
that $G$ must be a subset of at least one facet, say 
$F_i = \{x_i,\ldots,x_{i+k}\}$ with $1 \leq i \leq n-k$.  The
face $G$ then has one of three forms: (1) $G = \{x_{i+1},x_{i+2},\ldots,x_{i+k}\}$; (2) $G = \{x_i,x_{i+1},\ldots,x_{i+k-1}\}$ or 
(3) $G = \{x_i,\ldots,\hat{x}_j,\ldots,x_{i+k}\}$ with
$i < j < i+k$.

In the first case, there are at most two facets that
contain $x_{i+1}$ and $x_{i+k}$, namely $F_i$ and $F_{i+1}$ (it is
possible $F_{i+1}$ does not exist if $i=n-k$).  Similarly,
in the second case, only $F_{i-1}$ and $F_i$ contain $x_i$ and
$x_{i+k-1}$ (again, it is possible $F_{i-1}$ does not exist if $i=1$).
So in these cases, $G$ can only belong to at most two facets.  
In the third case, $F_i$ is the only facet to contain both
$x_i$ and $x_{i+k}$.  Thus, in all three cases, $G$ can belong
to at most two facets.

To prove the third requirement for pseudo-manifolds, 
order the facets as $F_1,F_2,\ldots,F_{n-k}$.  For this
ordering, $F_i \cap F_{i+1} =
\{x_{i+1},\ldots,x_{i+k}\}$ has dimension $k-1$. So for any two facets, $F_a, F_b$ with $a<b$, we can choose the sequence $F_a, F_{a+1}, F_{a+2}, $\ldots$, F_{a+l=b}$ to be the sequence we need to meet the third requirement.
Hence if $\frac{n}{2}$ $\leq$ $k$ $\leq$ $n$, then ${\tt vdw}(n,k)$ is a pseudo-manifold.    Finally, note that ${\tt vdw}(n,k)$ also
has a boundary in this case since the $(k-1)$-dimensional
face $\{x_1,x_2,\ldots,x_k\}$ only belongs to the facet $F_1$
if $\frac{n}{2} \leq k \leq n$.

To prove the forward direction, suppose 3 $\leq$ $k$ $<$ $\frac{n}{2}$.
Then two facets of ${\tt vdw}(n,k)$ are
$$
F_1 = \{x_1, x_3, x_5, \ldots, x_{2k+1}\}
~~\mbox{and}~~
F_2 = \{x_1,x_2,\ldots,x_{k+1}\}.$$
Note that in $F_1$ all the indices are odd, while 
in $F_2$, both odd and even indices appear.

Suppose ${\tt vdw}(n,k)$ is a pseudo-manifold. Then by 
Definition \ref{PMDEF} (3), there must be a sequence of facets 
$F_1 = G_0, G_1, \ldots, G_m = F_2$ such that the 
dimension of $G_i$ $\cap$ 
$G_{i+1}$ is $(k-1)$ for all $0 \leq i \leq m-1$. Since all the
indices of $F_1$ are odd, 
but $F_2$ does not have this property, 
there must be some facet $G_p$ in this chain that is the last to consist of only odd indices.  Since $G_p$ consists of all odd entries only, this implies that $G_p \cap G_{p+1}$ consists of $k$ elements, all
of which have an odd index. So $G_{p+1}$ has at least $k$ vertices 
with odd indices. Because $G_{p+1}$ is the first facet to have indices of mixed parity, of its $k+1$ elements, $k$ vertices will have an odd index and one vertex has an even index. 
We now have a contradiction to Lemma \ref{X1}, since
every facet must have at least $\left\lfloor {(k+1)/{2}} \right\rfloor$ $\geq$ $\left\lfloor {(3+1)/{2}} \right\rfloor$ = 2 even indices
if it does not contain all odd indices. Thus no such chain of facets can exist, so ${\tt vdw}(n,k)$ cannot be  a pseudo-manifold. 
\end{proof}

\begin{lemma}\label{lem.pseudomanifoldk=23}
${\tt vdw}(n,1)$ is a pseudo-manifold if and only if $n \leq 3$. ${\tt vdw}(n,2)$ is a pseudo-manifold if and only if $n \leq 6$.
\end{lemma}

\begin{proof}

\textit{Case 1: $k=1$.}  We break this into subcases:

(1) ${\tt vdw}(2,1)$ =  $\langle$$\{x_{1}, x_{2}\}$$\rangle$. This simplicial complex has only one facet and thus satisfies all three conditions to be a pseudo-manifold.

(2) ${\tt vdw}(3,1)$ = $\langle\{x_{1}, x_{2}\}, \{x_{2}, x_{3}\}, \{x_{1}, x_{3}\} \rangle$. Each vertex is contained in exactly two facets, and for any of the three facets, changing either of its vertices gives the other two facets, again satisfying all three conditions for it to be a pseudo-manifold.

(3) For $n$ $\geq$ 4, we claim ${\tt vdw}(n,1)$ is not a pseudo-manifold. For all such ${\tt vdw}(n,1)$, the following are valid facets:
\{$x_{1}$, $x_{2}$\}, \{$x_{1}$, $x_{3}$\}, and \{$x_{1}$, $x_{4}$\}.
Since $k=1$, $\{x_1\}$ is a $(k-1)$-dimensional face that
belongs to at least three facets,
violating Definition \ref{PMDEF} (2). Thus ${\tt vdw}(n,1)$ is not a pseudo-manifold for all $n \geq 4.$

\textit{Case 2: $k=2$}.  We break this into subcases.

 (1) ${\tt vdw}(3,2) = \langle \{x_{1}, x_{2}, x_{3}\}\rangle$.
It is immediate this complex is  a pseudo-manifold.

(2) ${\tt vdw}(4,2) = \langle \{x_{1}, x_{2}, x_{3}\}, 
\{x_{2}, x_{3}, x_{4}\} \rangle$. 
The one-dimensional face \{$x_{1}$, $x_{2}$\} appears in two facets, whereas all other one-dimensional faces only appear in one facet, thus satisfying Definition \ref{PMDEF} (2). The intersection of the two facets is \{$x_{1}$, $x_{2}$\}, which has dimension one.
So  ${\tt vdw}(4,2)$ is a pseudo-manifold.

($3$) ${\tt vdw}(5,2)$ is generated by the facets: 
$$F_0 = \{x_{1}, x_{2}, x_{3}\},~ F_1  =\{x_{2}, x_{3}, x_{4}\},~ 
F_2 =\{x_{3}, x_{4}, x_{5}\}, ~~ F_3 = \{x_{1}, x_{3}, x_{5}\}.$$ 
All one-dimensional faces are in one facet except for \{$x_{1}$, $x_{3}$\} , \{$x_{2}$, $x_{3}$\}, \{$x_{3}$, $x_{4}$\}, and \{$x_{3}$, $x_{5}$\}, which appear in two. Thus all 1-dimensional facets are contained in at most two facets. Finally, the sequence $F_0, F_1, F_2, F_3$ satisfies Definition \ref{PMDEF} (3), confirming ${\tt vdw}(5,2)$ is a pseudo-manifold.

 (4) ${\tt vdw}(6,2)$ is generated by the facets: 
\begin{multline*}
    F_0 = \{x_{1}, x_{2}, x_{3}\}, F_1  =\{x_{2}, x_{3}, x_{4}\}, 
F_2 =\{x_{3}, x_{4}, x_{5}\}, F_3 = \{x_{4}, x_{5}, x_{6}\}, \\ 
F_4 = \{x_{1}, x_{3}, x_{5}\}, F_5 = \{x_{2}, x_{4}, x_{6}\}.
\end{multline*}The 1-dimensional faces
\{$x_{1}$, $x_{3}$\}, \{$x_{2}$, $x_{3}$\}, \{$x_{2}$, $x_{4}$\}, \{$x_{3}$, $x_{4}$\}, \{$x_{3}$, $x_{5}$\}, \{$x_{4}$, $x_{5}$\}, and \{$x_{4}$, $x_{6}$\} appear in two facets, while all other 1-dimensional faces only appear in one facet, satisfying
Definition \ref{PMDEF} (2).  The sequence $F_4, F_0, F_1, F_2, F_3, F_5$ satisfies \ref{PMDEF} (3), resulting in ${\tt vdw}(6,2)$ also being a pseudo-manifold.

(5) For $n$ $\geq$ 7, the sets
\{$x_{3}$, $x_{4}$, $x_{5}$\}, \{$x_{1}$, $x_{3}$, $x_{5}$\}, and \{$x_{3}$, $x_{5}$, $x_{7}$\} are valid facets of ${\tt vdw}(n,2)$. The one-dimensional face \{$x_{3}$, $x_{5}$\} is contained within all three of these facets, violating Definition \ref{PMDEF} (2), thus
proving ${\tt vdw}(n,2)$ is not a pseudo-manifold for $n$ $\geq$ 7. 
\end{proof}

We summarize these results in the following corollary:
\begin{corollary}
\label{VDW} Let $n \geq k \geq 1$ be integers.  Then
    ${\tt vdw}(n,k)$ is a pseudo-manifold if and only if:
    \begin{enumerate}
        \item $k=1$, and $n=2$ or $3$, or
        \item $k=2$, and $n=3,4,5,$ or $6$, or
        \item $k \geq 3$, and $\frac{n}{2}$ $\leq$ k $<$ n.
    \end{enumerate}
\end{corollary}

Reviewing the proof of Lemma \ref{lem.pseudomanifoldk=23}, we
see that every pseudo-manifold has at least one $(k-1)$-dimensional
face that belongs to only one facet {\it except} ${\tt vdw}(3,1)$. 
Consequently, we also have the following characterization of
van der Waerden complexes that are pseudo-manifolds with boundary.

\begin{corollary}
\label{pseudobound}
Let $n \geq k \geq 1$ be integers. Then
${\tt vdw}(n,k)$ is a pseudo-manifold with boundary if and only if
${\tt vdw}(n,k)$ is a pseudo-manifold and $(n,k) \neq (3,1)$.
\end{corollary}

The relevance of pseudo-manifolds to the study of the WLP of
$A({\tt vdw}(n,k))$ comes via the following result
of Dao and Nair \cite{DN}.  In particular, 
if ${\tt vdw}(n,k)$ is a pseudo-manifold, then this implies
the WLP in certain degrees.  While we state the result
of Dao and Nair, we do not need the dual graph, so we do not define
it here.

\begin{theorem}[{\cite[Theorems 4.2 and 4.4]{DN}}]
\label{X3}
Suppose ${\rm char}(K) = 0$.
    Let $\Delta$ be a pseudo-manifold with $dim(\Delta) = d.$ Then $A(\Delta)$ has the WLP in degree $d$ if and only if:
    \begin{enumerate}
        \item the pseudo-manifold has a boundary, or
        \item the pseudo-manifold does not have a boundary, but its dual graph is not bipartite.
    \end{enumerate} 
\end{theorem}

Here is the relevant corollary for Artinian rings constructed
from ${\tt vdw}(n,k)$.  This is a direct consequence 
of Theorems \ref{X2} and \ref{X3}.

\begin{corollary}
\label{YASS}
Suppose ${\rm char}(K) = 0$ and let
     $n > k \geq 3$. Then $A({\tt vdw}(n,k))$ has the WLP in degree $k$ if $\frac{n}{2}$ $\leq$ $k$ $<$ $n$.
\end{corollary}

%\begin{proof}
%    By Corollaries $\ref{VDW}$, and $\ref{pseudobound}$ all such ${\tt vdw}(n,k)$ are pseudo-manifolds with boundary. Thus all of these pseudo-manifolds have the WLP in degree $k$.
%\end{proof}

\begin{remark}
    We could also deduce similar result for $A({\tt vdw}(n,1))$ 
    and $A({\tt vdw}(n,2))$ when ${\tt vdw}(n,1)$ and ${\tt vdw}(n,2)$
    are pseudo-manifolds,  However, these results are much weaker
    than Theorems \ref{KISONE} or \ref{KISTWO}, which already shows
    these Artinian rings have the WLP.
\end{remark}

%%%%%%%%%%%%%%%%%%%%%%%%%%%%%%%%%%%%%%%%%%%%%%%%%%%%%%%%

\section{Future directions}
\label{sec.futuredirections}

In this paper, we have completely determined when
$A({\tt vdw}(n,k))$ has the WLP for $k=1,2$ and $3$.  
Using Macaulay2, we computed when $A({\tt vdw}(n,k))$ has
the WLP for all $1 \leq k < n \leq 20$, in the case that
${\rm char}(K) = 0$. Our results
are summarized in Table \ref{TABLEWLP}.  
Note that the first three columns have been determined for
all $n$ by Theorems \ref{KISONE}, \ref{KISTWO}, and \ref{thm.casek=3},
while the ``diagonal'' is always true via Theorem \ref{NMINUS1}.

Based upon our computations, we propose the following conjecture.

\begin{conjecture}\label{conjecture}
    Suppose ${\rm char}(K) = 0$. Fix an integer $k \geq 3$.  
    \begin{enumerate}
    \item If $k$ is odd, 
    then there exists an $m$ such that  $A({\tt vdw}(n,k))$ fails to have
    the WLP for $n \geq m$. 
    \item If $k$ is even, then $A({\tt vdw}(n,k))$ has
    the WLP for all $n > k$.
    \end{enumerate}
\end{conjecture}

\noindent
In terms of Table \ref{TABLEWLP}, the conjecture says that in
any column indexed by an odd $k$ (except for $k=1$), one should eventually
expect to see an \xmark, i.e., $A({\tt vdw}(n,k))$ will fail to have  the
WLP for large $n$.    

If Conjecture \ref{conjecture} is true, then
for each odd integer $k>1$, it would also be 
interesting to determine the minimal $n$ for which
$A({\tt vdw}(n,k))$ fails to have the WLP.  In Table \ref{firstfailure} we have recorded the 
smallest $n$ for which $A({\tt vdw}(n,k))$ fails to have the WLP for $3 \leq k \leq 11$ (this was the limit of what we could test). 
Based upon this scant evidence,  it appears that 
the minimal
$n$ satisfies
$$ n = 2k+2 -\frac{(1+(-1)^{\frac{k+1}{2}})}{2} ~~\mbox{for odd $k \geq 3$.}$$

\begin{table}[h!]
\begin{tabular}{|c|*{19}{c|}}
\hline
\diagbox{$n$}{$k$} & 1 & 2 & 3 & 4 & 5 & 6 & 7 & 8 & 9 & 10 & 11 & 12 & 13 & 14 & 15 & 16 & 17 & 18 & 19 \\
\hline
%      1         2       3        4        5        6        7        8        9        10       11        12       13      14       15       16       17        18        19
2  & \cmark & \gmark & \gmark & \gmark & \gmark & \gmark & \gmark & \gmark & \gmark & \gmark & \gmark & \gmark & \gmark & \gmark & \gmark & \gmark & \gmark & \gmark & \gmark \\ \hline
3  & \cmark & \cmark & \gmark & \gmark & \gmark & \gmark & \gmark & \gmark & \gmark & \gmark & \gmark & \gmark & \gmark & \gmark & \gmark & \gmark & \gmark & \gmark & \gmark \\ \hline
4  & \cmark & \cmark & \cmark & \gmark & \gmark & \gmark & \gmark & \gmark & \gmark & \gmark & \gmark & \gmark & \gmark & \gmark & \gmark & \gmark & \gmark & \gmark & \gmark \\ \hline
5  & \cmark & \cmark & \cmark & \cmark & \gmark & \gmark & \gmark & \gmark & \gmark & \gmark & \gmark & \gmark & \gmark & \gmark & \gmark & \gmark & \gmark & \gmark & \gmark \\ \hline
6  & \cmark & \cmark & \cmark & \cmark & \cmark & \gmark & \gmark & \gmark & \gmark & \gmark & \gmark & \gmark & \gmark & \gmark & \gmark & \gmark & \gmark & \gmark & \gmark \\ \hline
7  & \cmark & \cmark & \xmark & \cmark & \cmark & \cmark & \gmark & \gmark & \gmark & \gmark & \gmark & \gmark & \gmark & \gmark & \gmark & \gmark & \gmark & \gmark & \gmark \\ \hline
8  & \cmark & \cmark & \xmark & \cmark & \cmark & \cmark & \cmark & \gmark & \gmark & \gmark & \gmark & \gmark & \gmark & \gmark & \gmark & \gmark & \gmark & \gmark & \gmark \\ \hline
9  & \cmark & \cmark & \xmark & \cmark & \cmark & \cmark & \cmark & \cmark & \gmark & \gmark & \gmark & \gmark & \gmark & \gmark & \gmark & \gmark & \gmark & \gmark & \gmark \\ \hline
10 & \cmark & \cmark & \xmark & \cmark & \cmark & \cmark & \cmark & \cmark & \cmark & \gmark & \gmark & \gmark & \gmark & \gmark & \gmark & \gmark & \gmark & \gmark & \gmark \\ \hline
11 & \cmark & \cmark & \xmark & \cmark & \cmark & \cmark & \cmark & \cmark & \cmark & \cmark & \gmark & \gmark & \gmark & \gmark & \gmark & \gmark & \gmark & \gmark & \gmark \\ \hline
12 & \cmark  & \cmark & \xmark & \cmark & \xmark & \cmark & \cmark & \cmark & \cmark & \cmark & \cmark & \gmark & \gmark & \gmark & \gmark & \gmark & \gmark & \gmark & \gmark \\ \hline
13 & \cmark & \cmark & \xmark & \cmark & \xmark & \cmark & \cmark & \cmark & \cmark & \cmark & \cmark & \cmark & \gmark & \gmark & \gmark & \gmark & \gmark & \gmark & \gmark \\ \hline
14 & \cmark & \cmark & \xmark & \cmark & \xmark & \cmark & \cmark & \cmark & \cmark & \cmark & \cmark & \cmark & \cmark & \gmark & \gmark & \gmark & \gmark & \gmark & \gmark \\ \hline
15 & \cmark & \cmark & \xmark & \cmark & \xmark & \cmark & \xmark & \cmark & \cmark & \cmark & \cmark & \cmark & \cmark & \cmark & \gmark & \gmark & \gmark & \gmark & \gmark \\ \hline
16 & \cmark & \cmark & \xmark & \cmark & \xmark & \cmark & \xmark & \cmark & \cmark & \cmark & \cmark & \cmark & \cmark & \cmark & \cmark  & \gmark & \gmark & \gmark & \gmark  \\ \hline
17 & \cmark & \cmark & \xmark & \cmark & \xmark & \cmark & \xmark & \cmark & \cmark & \cmark & \cmark & \cmark & \cmark & \cmark & \cmark & \cmark & \gmark & \gmark & \gmark \\ 
\hline
18 & \cmark & \cmark & \xmark & \cmark & \xmark & \cmark & \xmark & \cmark & \cmark & \cmark & \cmark & \cmark & \cmark & \cmark & \cmark & \cmark & \cmark & \gmark & \gmark \\ 
\hline
19 & \cmark & \cmark & \xmark & \cmark & \xmark & \cmark & \xmark & \cmark & \cmark & \cmark & \cmark & \cmark & \cmark & \cmark & \cmark & \cmark & \cmark & \cmark & \gmark \\ 
\hline
20 & \cmark & \cmark & \xmark & \cmark & \xmark & \cmark & \xmark & \cmark & \xmark & \cmark & \cmark & \cmark & \cmark & \cmark & \cmark & \cmark & \cmark & \cmark & \cmark \\ 
\hline
\end{tabular}\vspace{.5em}
\caption{The van der Waerden complexes ${\tt vdw}(n,k)$  
such that $A({\tt vdw}(n,k))$ has the WLP over a field $K$ where ${\rm char}(K) =0$ as verified by Macaulay2.
A \cmark~ indicates that the corresponding
Artinian ring has the WLP, while a \xmark~
indicates that the ring fails to have the WLP.}
\label{TABLEWLP}
\end{table}

\begin{table}[h!]
\begin{tabular}{cccccc}
\hline
\hline
$k$ & 3 & 5 & 7 & 9 & 11 \\
\hline 
\hline
\mbox{smallest $n$ where} & 7 & 12 & 15 & 20 & 23  \\
\mbox{$A({\tt vdw}(n,k))$ fails WLP} & & & & & \\
\hline
\hline
\end{tabular}
\caption{For odd $3 \leq k \leq 11$, the smallest $n$ for which
$A({\tt vdw}(n,k))$ fails to have the WLP.}
\label{firstfailure}
\end{table}
Our understanding of the SLP for $A({\tt vdw}(n,k))$ is much more 
limited.  Again, we have used Macaulay2 to determine 
when $A({\tt vdw}(n,k))$ has the SLP for all $1 \leq k < n \leq 14$.  Our results 
are summarized in Table \ref{SLPTABLE}.  The behaviour 
of the SLP is much more mysterious, although the table does suggest 
that $A({\tt vdw}(n,k))$ fails to have the SLP if $k \geq 3$ and $n \gg k$.
 It would be 
interesting to better understand what is happening in this situation.

As a final comment, both tables seem to indicate that 
$A({\tt vdw}(n,k))$ has the SLP and the WLP if $k$ is ``close'' to $n$, 
that is, $n-k$ is ``small''.  This is indeed true if $n-k=1$, as shown in
Theorem \ref{NMINUS1}.
It would be interesting to prove this observation.

\begin{table}[h]
\begin{tabular}{|c|*{14}{c|}}
\hline
\diagbox{$n$}{$k$} & 1 & 2 & 3 & 4 & 5 & 6 & 7 & 8 & 9 & 10 & 11 & 12 & 13 & 14\\
\hline
%      1         2       3        4        5        6        7        8        9        10       11        12       13      14       
2  & \cmark & \gmark & \gmark & \gmark & \gmark & \gmark & \gmark & \gmark & \gmark & \gmark & \gmark & \gmark & \gmark & \gmark\\ \hline
3  & \cmark & \cmark & \gmark & \gmark & \gmark & \gmark & \gmark & \gmark & \gmark & \gmark & \gmark & \gmark & \gmark & \gmark\\ \hline
4  & \cmark & \cmark & \cmark & \gmark & \gmark & \gmark & \gmark & \gmark & \gmark & \gmark & \gmark & \gmark & \gmark & \gmark\\ \hline
5  & \cmark & \xmark & \cmark & \cmark & \gmark & \gmark & \gmark & \gmark & \gmark & \gmark & \gmark & \gmark & \gmark & \gmark\\ \hline
6  & \cmark & \xmark & \cmark & \cmark & \cmark & \gmark & \gmark & \gmark & \gmark & \gmark & \gmark & \gmark & \gmark & \gmark\\ \hline
7  & \cmark & \xmark & \xmark & \cmark & \cmark & \cmark & \gmark & \gmark & \gmark & \gmark & \gmark & \gmark & \gmark & \gmark\\ \hline
8  & \cmark & \cmark & \xmark & \cmark & \cmark & \cmark & \cmark & \gmark & \gmark & \gmark & \gmark & \gmark & \gmark & \gmark\\ \hline
9  & \cmark & \cmark & \xmark & \xmark & \cmark & \cmark & \cmark & \cmark & \gmark & \gmark & \gmark & \gmark & \gmark & \gmark\\ \hline
10 & \cmark & \cmark & \xmark & \xmark & \cmark & \cmark & \cmark & \cmark & \cmark & \gmark & \gmark & \gmark & \gmark & \gmark\\ \hline
11 & \cmark & \cmark & \xmark & \xmark & \xmark & \cmark & \cmark & \cmark & \cmark & \cmark & \gmark & \gmark & \gmark & \gmark\\ \hline
12 & \cmark & \cmark & \xmark & \xmark & \xmark & \cmark & \cmark & \cmark & \cmark & \cmark & \cmark & \gmark & \gmark & \gmark\\ \hline
13 & \cmark & \cmark & \xmark & \xmark & \xmark & \xmark & \cmark & \cmark & \cmark & \cmark & \cmark & \cmark & \gmark & \gmark\\ \hline
14 & \cmark & \cmark & \xmark & \xmark & \xmark & \xmark & \cmark & \cmark & \cmark & \cmark & \cmark & \cmark & \cmark & \gmark\\ \hline
%15 & \cmark & \cmark & \xmark & \xmark & \xmark & \xmark & \xmark & \cmark & \cmark & \cmark & \cmark & \cmark & \cmark & \cmark & \gmark & \gmark \\ \hline
%16 & \cmark & \cmark & \xmark & \xmark & \xmark & \xmark & \xmark & \cmark & \cmark & \cmark & \cmark & \-- & \-- & \-- & \-- & \gmark \\ \hline
%17 & \cmark & \cmark & \xmark & \xmark & \xmark & \xmark & \xmark & \xmark & \cmark & \cmark & \-- & \-- & \-- & \-- & \-- & \-- \\ 
%\hline
%18 & \cmark & \cmark & \xmark & \xmark & \xmark & \xmark & \xmark & \xmark & \cmark & \cmark & \-- & \-- & \-- & \-- & \-- & \-- \\ 
%\hline
%19 & \cmark & \cmark & \xmark & \xmark & \xmark & \xmark & \xmark & \xmark & \cmark & \cmark & \-- & \-- & \-- & \-- & \-- & \-- \\ 
%\hline
\end{tabular}\vspace{.5em}
\caption{
The van der Waerden complexes ${\tt vdw}(n,k)$ 
such that $A({\tt vdw}(n,k))$ has the SLP  over a field $K$ where ${\rm char}(K) =0$, as verified by Macaulay2.
A \cmark~ indicates that the corresponding
Artinian ring has the SLP, while a \xmark~
indicates that the ring fails to have the SLP.}
\label{SLPTABLE}
\end{table}

%%%%%%%%%%%%%%%%%%%%%%%%%%%%%%%%%%%%%%%%%

\appendix

\section{Bounds on the $f$-vector of ${\tt vdw}(n,3)$}
\label{appendix}

In this appendix, we provide a proof of Lemma \ref{lem.f-vector-vdw(n,3)}. We begin with the following lemma, whose
importance will become apparent later in the appendix.

\begin{lemma}\label{lem.technicalidentity}
    Let $n \geq 88$ be an integer.  Set
    $e = \max\{j ~|~ 1 +4j \leq n \}$ and $d = \max\{j ~|~ 1 +3j \leq n\}$.  Then
    $2d-3e+5 \leq 0.$
\end{lemma}

\begin{proof}
    We can write $n$ as 
$$
    n  =  1  + 4e + \delta ~~\mbox{with $\delta \in \{0,1,2,3\}$, and }
    n  =  1 + 3d + \epsilon ~~\mbox{with $\epsilon \in \{0,1,2\}$.}
$$
Solving for $e$ and $d$ gives 
$$e = \frac{n-1-\delta}{4}~~\mbox{and}~~ d = \frac{n-1-\epsilon}{3}.$$
Thus 
\begin{eqnarray*}
    2d-3e+5 & = & 2\left(\frac{n-1-\epsilon}{3}\right) - 3
    \left(\frac{n-1-\delta}{4}\right) + 5 
     =  \frac{-n+61+ (9\delta-8\epsilon)}{12}.
\end{eqnarray*}
By considering all possible pairs $(\delta,\epsilon)$,
it follows that $9\delta-8\epsilon \leq 27$.
So 
$$  2d-3e+5 
     =  \frac{-n+61+ (9\delta-8\epsilon)}{12} \leq \frac{-n+88}{12}$$
Because $n \geq 88$, the conclusion now follows.
\end{proof}

\begin{lemma}[Lemma \ref{lem.f-vector-vdw(n,3)}]
Let $n$ be a integer such that $n \geq 8$.
If $f({\tt vdw}(n,3)) = (1,f_0,f_1,f_2,f_3)$ is the 
$f$-vector of ${\tt vdw}(n,3)$, then $f_1 \leq f_2$.
\end{lemma}

\begin{proof}
    By a computer computation, we can verify the statement
    for all $8 \leq n \leq 88$ (see Table \ref{PLS}).  So,
    we assume $n \geq 88$, and prove the result by induction
    on $n$.
    Let $\Delta(n) = {\tt vdw}(n,3)$, and let $f_i(\Delta(n))$
    denote the number of faces of dimension $i$ of $\Delta(n)$.
    
    Since an arithmetic progression of length three
    on $[n-1]$ is also an arithmetic progression of
    length three on $[n]$, every facet of $\Delta(n-1)$
    is also a facet of $\Delta(n)$.  Thus, the facets
    of $\Delta(n)$ that are not in $\Delta(n-1)$ 
    are precisely the facets of $\Delta(n)$ that contain $x_n$.
    Moreover, the number of faces of dimension $i$ of $\Delta(n)$
    is the number of faces of dimension $i$ of $\Delta(n-1)$
    plus all the new additional faces of dimension $i$ that come
    from the new facets.  More precisely,
    $$f_i(\Delta(n)) = f_i(\Delta(n-1)) + |\{ F \in \Delta(n) \setminus \Delta(n-1) ~|~ \dim F = i\}|.$$
    So, we need to compute (or bound) the number
    of new faces of dimension $i=1$ and $i=2$ that are contributed
    by the facets of $\Delta(n)$ that contain $x_n$.
    
    Setting $d = \max\{j ~|~ 1+3j \leq n\}$, there are 
    exactly $d$ facets of $\Delta(n)$ that contain $x_n$, namely
    $$F_j = \{x_{n-3j},x_{n-2j},x_{n-j},x_n\}~~\mbox{for 
    $j=1,\ldots,d$}.
    $$

    We first consider $i=2$, i.e., the value of $f_2(\Delta(n))$.
    Let $e = \max\{j ~|~ 1 + 4j \leq n\}$.  For each 
    $j =1,\ldots,d$, the facet $F_j$ contributes 
    four faces of dimension two to $\Delta(n)$:
    $$\{x_{n-3j},x_{n-2j},x_{n-j}\},~ 
    \{x_{n-3j},x_{n-j},x_{n}\},~ 
    \{x_{n-3j},x_{n-2j},x_{n}\},~ 
    \{x_{n-2j},x_{n-j},x_{n}\}.
    $$
    For $j=1,\ldots,e$, the two dimensional face $\{x_{n-3j},x_{n-2j},x_{n-j}\}$ also belongs to the facet
    $$G_j = \{x_{n-4j},x_{n-3j},x_{n-2j},x_{n-j}\}$$
    of $\Delta(n-1)$. Note $G_j$ is a facet because $1 \leq n-4e \leq n-4j$.  So, the facets $F_1,\ldots,F_e$ each contribute
    three new faces of dimension two to $\Delta(n)$.  
    
    On the other
    hand, the facets $F_{e+1},\ldots,F_d$ all contribute
    four new faces of dimension two.  Indeed, if 
    $\{x_{n-3j},x_{n-2j},x_{n-j}\}$ belonged to $\Delta(n-1)$,
    then $\{n-3j,n-2j,n-j\}$ would be part of an arithmetic progression of length $3$ on $[n-1]$.  But this would require
    either $\{n-4j,n-3j,n-2j,n-j\}$ or $\{n-3j,n-2j,n-j,n\}$ to
    be an arithmetic progression on $[n-1]$.  However, the first is
    not possible since $j > e$, so $1+4j > n$, i.e., $1 > n-4j$,
    and the second progression is not possible since $n \not\in[n-1]$. 
    Summarizing the above discussion, we thus have
    \begin{equation}\label{eq:formulaf_2}
    f_2(\Delta(n)) = f_2(\Delta(n-1))+ 3e + 4(d-e) = f_2(\Delta(n-1)) + 4d-e.
    \end{equation}

    We now find an upper bound on $f_1(\Delta(n))$.  
     For each 
    $j =1,\ldots,d$, the facet $F_j$ contributes 
    six faces of dimension one to $\Delta(n)$:
    $$\{x_{n-3j},x_{n-2j}\},~ \{x_{n-3j},x_{n-j}\},~~\{x_{n-2j},x_{n-j}\},~~
    \{x_{n-3j},x_n\},~~ \{x_{n-2j},x_n\},~~ \{x_{n-j},x_n\}.$$
    For $j=1,\ldots,e$, the one-dimensional faces
    $\{x_{n-3j},x_{n-2j}\}, \{x_{n-3j},x_{n-j}\},$ and 
    $\{x_{n-2j},x_{n-j}\}$ belong to the facet 
    $G_j = \{x_{n-4j},x_{n-3j},x_{n-2j},x_{n-j}\} \in \Delta(n-1)$.  So, the facets $F_1,\ldots,F_e$ each contribute
    at most three new one-dimensional faces. 

    We can further refine this bound.  If $j \in \{1,\ldots,e\}$
    and $j=3k$, then
    \begin{eqnarray*}
        \{x_{n-j},x_n\} = \{x_{n-3k},x_n\} 
        & \subseteq & F_k = \{x_{n-3k},x_{n-2k},x_{n-k},x_n\} \\
        \{x_{n-2j},x_n\} = \{x_{n-6k},x_n\} 
        & \subseteq & F_{2k} = \{x_{n-6k},x_{n-4k},x_{n-2k},x_n\}.
    \end{eqnarray*}
   That is, two of the one-dimensional faces of  $F_j = F_{3k}$ already appear in $F_k$ and $F_{2k}$, so $F_{3k}$
   contributes at most one new one-dimensional face to
   $\Delta(n)$.

   If $j \in \{1,\ldots,e\}$ is such that $j=2k$ and $3\nmid k$, then
   \begin{eqnarray*}
        \{x_{n-j},x_n\} = \{x_{n-2k},x_n\} 
        & \subseteq & F_{k} = \{x_{n-3k},x_{n-2k},x_{n-k},x_n\},
    \end{eqnarray*}
    that is, if $F_j = F_{2k}$ and $3\nmid k$, then
    $F_{2k}$ contributes at most two new one-dimensional
    faces to $\Delta(n)$.

    Set $\ell =\lfloor \frac{e}{3} \rfloor$ and partition
    $F_1,\ldots,F_e$ into the following sets:
    $$\{F_1,F_2,F_3\} \cup \{F_4,F_5,F_6\} \cup \cdots
    \{F_{3\ell-2},F_{3\ell-1},F_{3\ell}\} \cup H$$
    where 
    $$H = \begin{cases}
        \emptyset & \mbox{if $e = 3\ell$} \\
        \{F_e\} & \mbox{if $e = 3\ell +1$} \\
        \{F_{e-1},F_{e}\} & \mbox{If $e=3\ell+2$.}
    \end{cases}$$
    Each triple $\{F_{3k-2},F_{3k-1},F_{3k}\}$ contributes
    at most six new faces of dimension one to $\Delta(n)$, that
    is, the one new face from $F_{3k}$, two from the facet $F_{3k-1}$ 
    or $F_{3k-2}$ (whichever  has an even index), 
    and three from the remaining facet.  Note that
    $H$ contributes at most $0, 3$ or $5$ 
    faces of dimension one (the 5 comes from the fact that either $e-1$ or $e$ is
    divisible by two, so one of $F_{e-1}$ and
    $F_{e}$
    contributes at
    most two new one-dimensional faces).
    So $F_1,\ldots,F_e$ contribute at most
    $6\ell + 5$ facets of dimension one.  

    Because the facets $F_{e+1},\ldots,F_d$ 
    each can contribute at most six facets of dimension one to $\Delta(n)$, we have 
    \begin{equation}\label{eqn.formula_f1}
    f_1(\Delta(n)) \leq f_1(\Delta(n-1))+ 6\ell+5 + 6(d-e) \leq f_1(\Delta(n-1))+ 6d -4e +5.
    \end{equation}
    where we used the fact that $\ell \leq \frac{e}{3}$.

    Since $n \geq 88$, by Lemma \ref{lem.technicalidentity}
    $$2d-3e+5 \leq 0 ~~\Leftrightarrow 6d-4e+5 \leq 4d-e.$$
    By induction $f_1(\Delta(n-1)) \leq f_2(\Delta(n-1))$.  Thus, by using 
    \eqref{eq:formulaf_2} and \eqref{eqn.formula_f1} and the above inequality,
    we get
    \begin{eqnarray*}
        f_1(\Delta(n)) & \leq & 
        f_1(\Delta(n-1)) + 6d-4e+5\\
        & \leq & f_2(\Delta(n-1))+4d-e = f_2(\Delta(n)),
    \end{eqnarray*}
    thus completing the proof.
\end{proof}

\begin{table}[ht]
\centering
\scriptsize
\setlength{\tabcolsep}{4pt}

\begin{tabular}{|l|r|r||l|r|r||l|r|r|}
\hline
$n$ & $\dim_K A_2$ & $\dim_K A_3$ &
$n$ & $\dim_K A_2$ & $\dim_K A_3$ &
$n$ & $\dim_K A_2$ & $\dim_K A_3$ \\
\hline

$A({\tt vdw}(4,3))$  & 6   & 4   &
$A({\tt vdw}(33,3))$ & 401 & 540 &
$A({\tt vdw}(62,3))$ & 1465 & 1990 \\

$A({\tt vdw}(5,3))$  & 9 & 7   &
$A({\tt vdw}(34,3))$ & 428 & 576 &
$A({\tt vdw}(63,3))$ & 1514 & 2055 \\

$A({\tt vdw}(6,3))$  & 12 & 10   &
$A({\tt vdw}(35,3))$ & 455 & 612 &
$A({\tt vdw}(64,3))$ & 1564 & 2124\\

$A({\tt vdw}(7,3))$  & 18 & 17   &
$A({\tt vdw}(36,3))$ & 482 & 648 &
$A({\tt vdw}(65,3))$ & 1613 & 2192 \\

$A({\tt vdw}(8,3))$  & 24 & 24   &
$A({\tt vdw}(37,3))$ & 510 & 687 &
$A({\tt vdw}(66,3))$ & 1662 & 2260 \\

$A({\tt vdw}(9,3))$  & 29 & 30   &
$A({\tt vdw}(38,3))$ & 538 & 726 &
$A({\tt vdw}(67,3))$ & 1714 & 2332 \\

$A({\tt vdw}(10,3))$  & 35 & 40   &
$A({\tt vdw}(39,3))$ & 566 & 765 &
$A({\tt vdw}(68,3))$ & 1766 & 2404\\

$A({\tt vdw}(11,3))$  & 41 & 50   &
$A({\tt vdw}(40,3))$ & 600 & 808 &
$A({\tt vdw}(69,3))$ & 1815 & 2475 \\

$A({\tt vdw}(12,3))$  & 47 & 60   &
$A({\tt vdw}(41,3))$ & 633 & 850 &
$A({\tt vdw}(70,3))$ & 1870 & 2550 \\

$A({\tt vdw}(13,3))$  & 56 & 73   &
$A({\tt vdw}(42,3))$ & 666 & 892 &
$A({\tt vdw}(71,3))$ & 1925 & 2625\\

$A({\tt vdw}(14,3))$  & 65 & 86   &
$A({\tt vdw}(43,3))$ & 702 & 938 &
$A({\tt vdw}(72,3))$ & 1980 & 2700 \\

$A({\tt vdw}(15,3))$  & 74 & 99   &
$A({\tt vdw}(44,3))$ & 738 & 984 &
$A({\tt vdw}(73,3))$ & 2036 & 2778\\

$A({\tt vdw}(16,3))$ & 89 & 116   &
$A({\tt vdw}(45,3))$ & 771 & 1029 &
$A({\tt vdw}(74,3))$ & 2092 & 2856 \\

$A({\tt vdw}(17,3))$  & 103 & 132   &
$A({\tt vdw}(46,3))$ & 805 & 1078 &
$A({\tt vdw}(75,3))$ & 2148 & 2934\\

$A({\tt vdw}(18,3))$  & 117 & 148   &
$A({\tt vdw}(47,3))$ & 839 & 1127 &
$A({\tt vdw}(76,3))$ & 2210 & 3016\\

$A({\tt vdw}(19,3))$  & 132 & 168   &
$A({\tt vdw}(48,3))$ & 873 & 1176 &
$A({\tt vdw}(77,3))$ & 2269 & 3097\\

$A({\tt vdw}(20,3))$  & 147 & 188   &
$A({\tt vdw}(49,3))$ & 910 & 1228 &
$A({\tt vdw}(78,3))$ & 2328 & 3178\\

$A({\tt vdw}(21,3))$  & 159 & 207   &
$A({\tt vdw}(50,3))$ & 947 & 1280 &
$A({\tt vdw}(79,3))$ & 2390 & 3263\\

$A({\tt vdw}(22,3))$  & 177 & 230   &
$A({\tt vdw}(51,3))$ & 984 & 1332 &
$A({\tt vdw}(80,3))$& 2452 & 3348 \\

$A({\tt vdw}(23,3))$  & 195 & 253   &
$A({\tt vdw}(52,3))$ & 1027 & 1388 &
$A({\tt vdw}(81,3))$ & 2513 & 3432\\

$A({\tt vdw}(24,3))$  & 213 & 276   &
$A({\tt vdw}(53,3))$ & 1067 & 1443 &
$A({\tt vdw}(82,3))$ & 2575 & 3520\\

$A({\tt vdw}(25,3))$  & 234 & 302   &
$A({\tt vdw}(54,3))$ & 1107 & 1498 &
$A({\tt vdw}(83,3))$ & 2637 & 3608\\

$A({\tt vdw}(26,3))$  & 255 & 328   &
$A({\tt vdw}(55,3))$ & 1148 & 1557 &
$A({\tt vdw}(84,3))$ & 2699 & 3696\\

$A({\tt vdw}(27,3))$  & 276 & 354   &
$A({\tt vdw}(56,3))$ & 1189 & 1616 &
$A({\tt vdw}(85,3))$ & 2764 & 3787\\

$A({\tt vdw}(28,3))$  & 298 & 384   &
$A({\tt vdw}(57,3))$ & 1229 & 1674  &
$A({\tt vdw}(86,3))$ & 2829 & 3878\\

$A({\tt vdw}(29,3))$  & 317 & 413   &
$A({\tt vdw}(58,3))$ & 1275 & 1736 &
$A({\tt vdw}(87,3))$ & 2894 & 3969\\

$A({\tt vdw}(30,3))$  & 336 & 442   &
$A({\tt vdw}(59,3))$ & 1321 & 1798 &
$A({\tt vdw}(88,3))$ & 2965 & 4064\\

$A({\tt vdw}(31,3))$  & 358 & 475   &
$A({\tt vdw}(60,3))$ & 1367 & 1860 &
$A({\tt vdw}(89,3))$ & 3035 & 4158\\

$A({\tt vdw}(32,3))$  & 380 & 508   &
$A({\tt vdw}(61,3))$ & 1416 & 1925 &
$A({\tt vdw}(90,3))$ & 3105 & 4252\\

\hline
\end{tabular}
\caption{Comparing $\dim_K A_2$ and $\dim_K A_3$ when $A = A({\tt vdw}(n,3))$ for $n=4,\ldots,90$. }
\label{PLS}
\end{table}

\section{Matrices used in Theorem \ref{thm.casek=3}}
\label{appendixB}

In the proof of Theorem \ref{thm.casek=3}, checking 
that $A({\tt vdw}(5,3))$ and $A({\tt vdw}(6,3))$ has the WLP
can be reduced to checking if the matrix representation
of the map $\times \ell$ has maximal rank (see Example \ref{ex.lingalgsetup}).  

For completeness, here are 
    the two matrices used to check that $A({\tt vdw}(5,3))$ 
    has the WLP in degrees two and three:
\begin{center}
\begin{blockarray}{cccccccccc}
$x_1x_2$ & $x_1x_3$ & $x_1x_4$ & $x_2x_3$ & $x_2x_4$ & $x_2x_5$ & $x_3x_4$ & $x_3x_5$ & $x_4x_5$\\
\begin{block}{[ccccccccc]c}
1 & 1 & 0 & 1 & 0 & 0 & 0 & 0 & 0 & $x_1x_2x_3$\\
1 & 0 & 1 & 0 & 1 & 0 & 0 & 0 & 0 & $x_1x_2x_4$\\
0 & 1 & 1 & 0 & 0 & 0 & 1 & 0 & 0 & $x_1x_3x_4$\\
0 & 0 & 0 & 1 & 1 & 0 & 1 & 0 & 0 & $x_2x_3x_4$\\
0 & 0 & 0 & 1 & 0 & 1 & 0 & 1 & 0 & $x_2x_3x_5$\\
0 & 0 & 0 & 0 & 1 & 1 & 0 & 0 & 1 & $x_2x_4x_5$\\
0 & 0 & 0 & 0 & 0 & 0 & 1 & 1 & 1 & $x_3x_4x_5$\\
\end{block}
\end{blockarray}
\end{center}
and
\[
\begin{blockarray}{ccccccc c}
x_1x_2x_3 & x_1x_2x_4 & x_1x_3x_4 & x_2x_3x_4 & x_2x_3x_5 & x_2x_4x_5 & x_3x_4x_5 & \\
\begin{block}{[ccccccc]c}
1 & 1 & 1 & 1 & 0 & 0 & 0 & x_1x_2x_3x_4\\
0 & 0 & 0 & 1 & 1 & 1 & 1 & x_2x_3x_4x_5\\
\end{block}
\end{blockarray}
\]
Both of these matrices have maximal rank, namely $7$ and $2$, 
respectively, in characteristic $0$.  (Note that the first matrix has rank 5 in characteristic
2, but we are avoiding this characteristic.)

The matrices for the case $n=6$ are 
\footnotesize
\[
\begin{blockarray}{cccccccccccc c}
x_1x_2 & x_1x_3 & x_1x_4 & 
x_2x_3 & x_2x_4 & x_2x_5 & 
x_3x_4 & x_3x_5 & x_3x_6 &
x_4x_5 & x_4x_6 & x_5x_6 & \\
\begin{block}{[cccccccccccc]c}
1 & 1 & 0 & 1 & 0 & 0 & 0 & 0 & 0 & 0 & 0 & 0 & x_1x_2x_3\\
1 & 0 & 1 & 0 & 1 & 0 & 0 & 0 & 0 & 0 & 0 & 0 & x_1x_2x_4\\
0 & 1 & 1 & 0 & 0 & 0 & 1 & 0 & 0 & 0 & 0 & 0 & x_1x_3x_4\\
0 & 0 & 0 & 1 & 1 & 0 & 1 & 0 & 0 & 0 & 0 & 0 & x_2x_3x_4\\
0 & 0 & 0 & 1 & 0 & 1 & 0 & 1 & 0 & 0 & 0 & 0 & x_2x_3x_5\\
0 & 0 & 0 & 0 & 1 & 1 & 0 & 0 & 0 & 1 & 0 & 0 & x_2x_4x_5\\
0 & 0 & 0 & 0 & 0 & 0 & 1 & 1 & 0 & 1 & 0 & 0 & x_3x_4x_5\\
0 & 0 & 0 & 0 & 0 & 0 & 1 & 0 & 1 & 0 & 1 & 0 & x_3x_4x_6\\
0 & 0 & 0 & 0 & 0 & 0 & 0 & 1 & 1 & 0 & 0 & 1 & x_3x_5x_6\\
0 & 0 & 0 & 0 & 0 & 0 & 0 & 0 & 0 & 1 & 1 & 1 & x_4x_5x_6\\
\end{block}
\end{blockarray}
\]
\normalsize
and
\footnotesize
\[
\begin{blockarray}{cccccccccc c}
x_1x_2x_3 & x_1x_2x_4 & x_1x_3x_4 & x_2x_3x_4 & x_2x_3x_5 & x_2x_4x_5 & x_3x_4x_5 & x_3x_4x_6 & x_3x_5x_6 & x_4x_5x_6 & \\
\begin{block}{[cccccccccc]c}
1 & 1 & 1 & 1 & 0 & 0 & 0 & 0 & 0 & 0 & x_1x_2x_3x_4\\
0 & 0 & 0 & 1 & 1 & 1 & 1 & 0 & 0 & 0 & x_2x_3x_4x_5\\
0 & 0 & 0 & 0 & 0 & 0 & 1 & 1 & 1 & 1 & x_3x_4x_5x_6\\
\end{block}
\end{blockarray}
\]
\normalsize
Again, both of these matrices have maximal rank, namely 10 and 3, so $A({\tt vdw}(6,3))$ has the WLP. 

%%%%%%%%%%%%%%%%%%%%%%%%%%%%%%%%%%%%%%%%%
\subsection*{Acknowledgments}
We thank the referee for their helpful comments and suggestions.
Some of the results in this paper first appeared in the MSc thesis of
Ragunathan \cite{R2025}.  We thank Brett Nasserden
for useful discussions.
This research was enabled in part by support provided by 
Compute Ontario ({\tt computeontario.ca})
and the Digital Research Alliance of Canada ({\tt alliance.can.ca}).
All computations were performed using \textit{Macaulay2}~\cite{M2}.
Van Tuyl's research is supported by NSERC Discovery Grant 2024-05299.
%%%%%%%%%%%%%%%%%%%%%%%%%%%%%%%%%%%%%%%%%

\end{document}